\documentclass[11pt,reqno]{amsart}
\usepackage[utf8]{inputenc}
\usepackage[T1]{fontenc}
\usepackage{amsmath}
\usepackage{mathtools}
\usepackage{bm}
\usepackage{xcolor}
\usepackage{graphicx}
\usepackage{amssymb}
\usepackage{braket}
\usepackage{enumitem}
\usepackage{dsfont}
\usepackage{hyperref}

\newtheorem{conjecture}{Conjecture}
\newtheorem{theorem}{Theorem}
\newtheorem{corollary}[theorem]{Corollary}
\newtheorem{lemma}{Lemma}[section]

\newtheorem{remark}[lemma]{Remark}

\numberwithin{equation}{section}

\newcommand{\dps}{\displaystyle}
\newcommand{\ii}{\infty}
\newcommand\R{{\ensuremath {\mathbb R} }}
\newcommand\C{{\ensuremath {\mathbb C} }}

\newcommand\1{{\ensuremath {\mathds 1} }}
\newcommand\bS{\mathbb S} 
\renewcommand\phi{\varphi}

\newcommand{\cE}{\mathcal{E}}

\newcommand{\cL}{\mathcal{L}}
\newcommand{\eps}{\epsilon}
\renewcommand{\epsilon}{\varepsilon}
\newcommand\pscal[1]{{\ensuremath{\left\langle #1 \right\rangle}}}
\newcommand{\norm}[1]{ \left| \! \left| #1 \right| \! \right| }

\renewcommand{\ge}{\geqslant}
\renewcommand{\le}{\leqslant}
\renewcommand{\geq}{\geqslant}
\renewcommand{\leq}{\leqslant}

\renewcommand{\tilde}{\widetilde}

\newcommand{\nn}{\nonumber}
\newcommand{\re}{{\rm Re}}
\newcommand{\im}{{\rm Im}}

\begin{document}
\title[The double-power nonlinear Schr\"odinger equation]{The double-power nonlinear Schr\"odinger equation and its generalizations: uniqueness, non-degeneracy and applications}

\author[M. Lewin]{Mathieu LEWIN}
\address{CNRS \& CEREMADE, Universit\'e Paris-Dauphine, PSL Research University, Place de Lattre de Tassigny, 75016 Paris, France.}
\email{mathieu.lewin@math.cnrs.fr}

\author[S. Rota Nodari]{Simona ROTA NODARI}
\address{Institut de Math\'ematiques de Bourgogne (IMB), CNRS, UMR 5584,
Universit\'e Bourgogne Franche-Comt\'e,
21000 Dijon, France.}
\email{simona.rota-nodari@u-bourgogne.fr}

\date{\today}

\begin{abstract}
In this paper we first prove a general result about the uniqueness and non-degeneracy of positive radial solutions to equations of the form $\Delta u+g(u)=0$. Our result applies in particular to the double power non-linearity where $g(u)=u^q-u^p-\mu u$ for $p>q>1$ and $\mu>0$, which we discuss with more details. In this case, the non-degeneracy of the unique solution $u_\mu$ allows us to derive its behavior in the two limits $\mu\to0$ and $\mu\to\mu_*$ where $\mu_*$ is the threshold of existence. This gives the uniqueness of energy minimizers at fixed mass in certain regimes. We also make a conjecture about the variations of the $L^2$ mass of $u_\mu$ in terms of $\mu$, which we illustrate with numerical simulations. If valid, this conjecture would imply the uniqueness of energy minimizers in all cases and also give some important information about the orbital stability of $u_\mu$.
\end{abstract}

\maketitle

\vspace{-1cm}
\tableofcontents

\section{Introduction}

In this paper we study the positive solutions (`ground states') to some semi-linear elliptic equations in $\R^d$  of the general form
\begin{equation}\label{nleq}
\boxed{\Delta u +g(u)=0\qquad \text{ in $\R^d$ with $d\geq2$},}
\end{equation}
where $\Delta u=\sum_{i=1}^d\partial^2_{x_i}u$ is the Laplacian. We assume that the nonlinearity $g$ is gauge-invariant under the action of the group $\bS^1$, that is
\begin{equation}\label{phaseinv}
g(|u|e^{i\theta})=g(|u|)e^{i\theta}
\end{equation} 
for any $u,\theta\in \R$. In other words, without loss of generality we may assume that $g:\R\to\R$ is a real odd function such that $g(0)=0$. Then~\eqref{nleq} is invariant under translations and multiplications by a phase factor. 

The study of the \emph{existence} and \emph{uniqueness} of positive solutions to equations of the type~\eqref{nleq} has a very long history. Of particular interest is the (focusing) \emph{nonlinear Schr\"odinger equation} (NLS) corresponding to 
\begin{equation}
 g(u)=u^q-u,\qquad 1<q<2^*-1
 \label{eq:g_NLS}
\end{equation}
where $2^*=2d/(d-2)$ is the critical Sobolev exponent in dimensions $d\geq3$ and $2^*=\ii$ in dimensions $d=1,2$. Here and everywhere else in the paper we use the convention that $u^q:=|u|^{q-1}u$ to ensure that~\eqref{phaseinv} is satisfied. 
In the particular case~\eqref{eq:g_NLS}, the uniqueness of positive solutions was proved first by Coffman~\cite{Coffman-72} for $q=3$ and $d=3$, and then by Kwong~\cite{Kwong-89} in the general case. These results have been extended to a larger class of non-linearities by many authors, including for instance~\cite{PelSer-83,LeoSer-87,KwoZha-91,CheLin-91,McLeod-93,PucSer-98,SerTan-00,Jang-10}. 

Another important property for applications is the \emph{non-degeneracy} of these solutions, which means that the kernel of the linearized operators is trivial, modulo phase and space translations:
$$\ker\left(\Delta+\frac{g(u)}{u}\right)={\rm span}\{u\},\qquad \ker\left(\Delta+g'(u)\right)={\rm span}\{\partial_{x_1}u,...,\partial_{x_d}u\}.$$
This property plays a central role for the stability or instability of these stationary solutions~\cite{Weinstein-85,ShaStr-85,GriShaStr-87,GriShaStr-90} in the context of the time-dependent Schr\"odinger equation
$$i\partial_tu=\Delta u+g(u).$$
In the NLS case~\eqref{eq:g_NLS} non-degeneracy was shown first in~\cite{Coffman-72,Kwong-89,Weinstein-85}, but for general nonlinearities it does not necessarily follow from the method used to show uniqueness. 

In this paper, we are particularly interested in the \emph{double-power nonlinearity}
\begin{equation}
\boxed{g(u)=-u^p+u^q-\mu u,\qquad p>q>1,\quad\mu>0,\quad d\geq2}
\label{eq:g_double_power_NLS}
\end{equation}
with a defocusing large exponent $p$ and a focusing smaller exponent $q$. Uniqueness in this case was shown in~\cite{SerTan-00}, but the non-degeneracy of the solutions does not seem to follow from the proof. The nonlinearity~\eqref{eq:g_double_power_NLS} has very strong physical motivations. The cubic-quintic nonlinearity $q=3,p=5$ appears in many applications and usually models systems with attractive two-body interactions and repulsive three-body interactions~\cite{Anderson-71,Kartavenko-84,MerIsi-88}.  In this case, non-degeneracy was shown in~\cite{KilOhPocVis-17} for $d=3$ and in~\cite{CarSpa-20_ppt} for $d=2$. On the other hand, the case $p=7/3$ and $q=5/3$ in dimension $d=3$ was considered in~\cite{Ricaud-17} in the context of symmetry breaking for a model of Density Functional Theory for solids. A general result which covers~\eqref{eq:g_double_power_NLS} appeared later in~\cite{AdaShiWat-18}.

Some time before~\cite{KilOhPocVis-17,Ricaud-17,AdaShiWat-18}, we had considered in~\cite{LewRot-15} the case 
$$g(u)=a\sin^{3}u\cos u-b \sin u\cos u,\qquad a>2b,$$ 
which naturally arises in the non-relativistic limit of a Dirac equation in nuclear physics~\cite{Rota-PhD,EstRot-12,EstRot-13,TreRot-13}. As was already mentioned in~\cite{LewRot-15}, the proof we gave of the uniqueness and non-degeneracy of radial solutions in this particular case is general and can be applied to a variety of situations. It covers the double-power nonlinearity~\eqref{eq:g_double_power_NLS} in the whole possible range of the parameters. In order to clarify the situation, in Section~\ref{sec:uniqueness_general} we start by reformulating the result of~\cite{LewRot-15} in a general setting. We also provide a self-contained proof in Section~\ref{sec:proof_thm_uniqueness} for the convenience of the reader.

Next we discuss at length the properties of the unique and non-degenerate solution $u_\mu$ for the double-power non-linearity~\eqref{eq:g_double_power_NLS}. Of high interest is the $L^2$ mass 
$$M(\mu)=\int_{\R^d}u_\mu(x)^2\,dx$$
of this solution. In the NLS case~\eqref{eq:g_NLS} the mass is a simple explicit power of $\mu$, but for the double-power nonlinearity~\eqref{eq:g_double_power_NLS}, $M$ is an unknown function. In Section~\ref{sec:double_power} we determine its exact behavior in the two regimes $\mu\to0^+$ and $\mu\to\mu_*^-$, where $\mu_*$ is the threshold for existence of solutions. This allows us to make an important conjecture about the variations of $M$ over the whole interval $(0,\mu_*)$, partly inspired by~\cite{KilOhPocVis-17,Ricaud-17,CarSpa-20_ppt} and supported by numerical simulations. In short, our conjecture says that the branch of solutions has at most one unstable part, that is, $M'$ vanishes at most once over $(0,\mu_*)$. 

One important motivation for studying the variations of $M$ concerns the uniqueness of energy minimizers at fixed mass
\begin{multline}
I(\lambda)=\inf\bigg\{ \frac12\int_{\R^d}|\nabla u|^2\,dx+\frac{1}{p+1}\int_{\R^d} |u|^{p+1}\,dx- \frac{1}{q+1}\int_{\R^d} |u|^{q+1}\,dx\ :\\
u\in H^1(\R^d)\cap L^{p+1}(\R^d),\ \int_{\R^d}|u|^2\,dx=\lambda\bigg\},
\label{eq:I_lambda_intro}
\end{multline}
which naturally appears in physical applications. Any minimizer, when it exists, is positive and solves the double-power NLS equation for some Lagrange multiplier $\mu$, hence equals $u_\mu$ after an appropriate space translation. The difficulty here is that several $\mu$'s could in principle give the same mass $\lambda$ and the same energy $I(\lambda)$, so that the uniqueness of solutions to the equation at fixed $\mu$ does not at all imply the uniqueness of energy minimizers. Nevertheless we conjecture that minimizers of~\eqref{eq:I_lambda_intro} are always unique. This would actually follow from the previously mentioned conjecture on $M$, since the latter implies that $M$ is one-to-one on the unique branch of stable solutions. In fact, our analysis of the function $M$ allows us to prove some partial results which, we think, are an interesting first step towards a better understanding of the general case. More precisely, we show that
\begin{itemize}
 \item minimizers of $I(\lambda)$ are always unique for $\lambda$ large enough and for $\lambda$ close enough to the critical mass $\lambda_c$ above which minimizers exist;
 \item the set of $\lambda$'s for which minimizers are not unique is at most finite;
 \item the number of minimizers at those exceptional values of $\lambda$ is also finite, modulo phases and space translations. 
\end{itemize}

\medskip

The paper is organized as follows. In the next section we state our main result on the uniqueness and non-degeneracy of solutions to~\eqref{nleq}. Then, in Section~\ref{sec:double_power} we describe our findings on the double-power NLS equation. The rest of the paper is devoted to the proof of our main results. 

\bigskip

\noindent{\textbf{Acknowledgement.}} We thank R\'emi Carles and Christof Sparber for useful discussions. This project has received funding from the European Research Council (ERC) under the European Union's Horizon 2020 research and innovation programme (grant agreement MDFT No 725528 of M.L.) and from the Agence Nationale de la Recherche (grant agreement DYRAQ, ANR-17-CE40-0016, of S.R.N.).

\section{Uniqueness and non-degeneracy: an abstract result }\label{sec:uniqueness_general}

Here we state our abstract result about the uniqueness and non-degeneracy of solutions of~\eqref{nleq}.

\begin{theorem}[Uniqueness and non-degeneracy]\label{thmuniqnondeg} 
Let $0< \alpha<\beta$ and $g$ be a continuously differentiable function on $[0,\beta]$. Assume that the following conditions hold:
\begin{enumerate}[label=\textnormal{(H\arabic*)}]
\item \label{hyp1} We have $g(0)=g(\alpha)=g(\beta)=0$, $g$ is negative on $(0,\alpha)$ and positive on $(\alpha,\beta)$ with $g'(0)<0$, $g'(\alpha)>0$ and $g'(\beta)\leq0$.

\smallskip

\item \label{hyp2} For every $\lambda>1$, the function 
\begin{equation}\label{defI}
I_{\lambda}(x):=xg'(x)-\lambda g(x)
\end{equation}
has exactly one root $x_*$ on the interval $(0,\beta)$, which belongs to $(\alpha,\beta)$. 
\end{enumerate}
Then equation \eqref{nleq} admits \emph{at most one} positive radial solution with $\|u\|_{\infty}<\beta$ and such that $u(x),u'(x)\to 0$ when $|x|\to \infty$. Moreover, when it exists, this solution is non-degenerate in the sense that 
\begin{equation}
\ker\big(\Delta+g'(u)\big)={\rm span}\left\{\partial_{x_1}u,\ldots,\partial_{x_d}u\right\},\qquad \ker\left(\Delta+\frac{g(u)}{u}\right)={\rm span}\{u\}.
\label{eq:non_degenerate}
\end{equation}
\end{theorem}

\begin{remark} The assumption~\ref{hyp2} can be replaced by the two stronger conditions
\begin{enumerate}
\item[\textnormal{(H2')}] there exists $x^*\in (0,\beta)$ such that $g''>0$ on $(0,x^*)$ and $g''<0$ on $(x^*,\beta)$;
\item[\textnormal{(H2'')}]  the function $x\mapsto \frac{xg'(x)}{g(x)}$ is strictly decreasing on $(\alpha,\beta)$.
\end{enumerate}
\end{remark}

\begin{remark}\label{rmk:Ilambda}
In the proof we use \textnormal{(H2)} only for one (unknown) particular $\lambda>1$. Should one be able to localize better this $\lambda$ for a concrete nonlinearity $g$, one would then only need to verify \textnormal{(H2)} in this region. 
\end{remark}

\begin{remark}\label{rmk:beta}
If $g$ is defined on the half line $\R_+$ and negative on $(\beta,\ii)$, then all the positive solutions satisfy $u<\beta$. This follows from the maximum principle since $-\Delta u=g(u)\leq0$ on $\{u\geq\beta\}$.
\end{remark}

\begin{figure}[t]
\centering
\includegraphics[width=8cm]{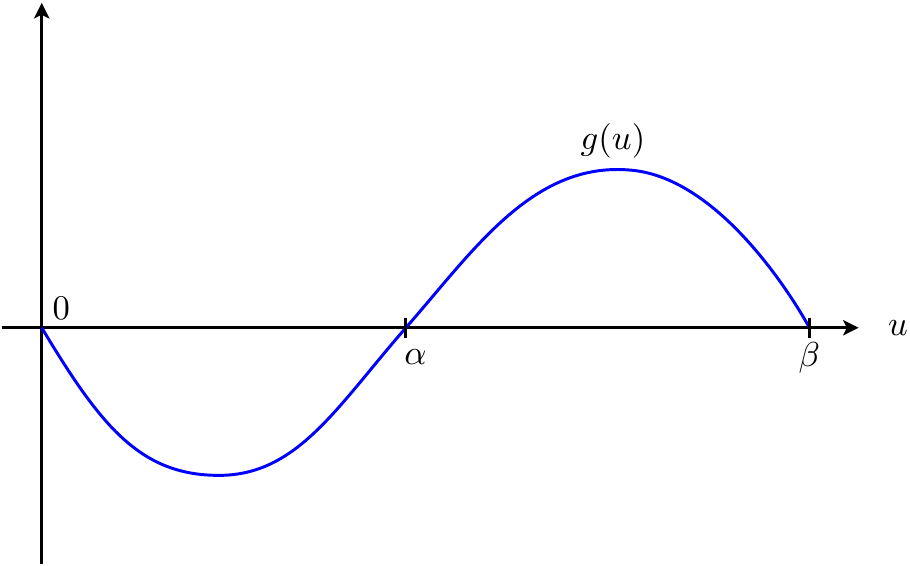}
\caption{Typical form of the admissible nonlinearity $g$ in~\eqref{nleq}.\label{fig:g}}
\end{figure}

As we have mentioned, Theorem~\ref{thmuniqnondeg} was indeed proved in~\cite{LewRot-15} although perhaps only implicitly since we were there mainly concerned with a special case for $g$ (see~\cite[Lemma~3]{LewRot-15} and the comments after the statement). The detailed proof of Theorem~\ref{thmuniqnondeg} is provided in Section~\ref{sec:proof_thm_uniqueness}, for the convenience of the reader. 

The operators appearing in~\eqref{eq:non_degenerate} are the two linearized operators (for the real and imaginary parts of $u$, respectively) at the solution $u$ and the non-degeneracy means that their kernel (at the solution $u$) is spanned by the generators of the two symmetries of the problem (space translations and multiplication by a phase). Non-degeneracy is important in many respects, as we will recall below.

Note that our assumptions (H1)--(H2) require the existence of three successive zeroes for $g$ as in Figure~\ref{fig:g}. In the traditional NLS case there are only two and this corresponds to taking $\beta=+\ii$ in our theorem, a situation where the same result is valid, as proved by McLeod in~\cite{McLeod-93} and reviewed in~\cite{Tao-06,Frank-13}. 

In the article~\cite{SerTan-00} by Serrin and Tang, uniqueness is proved under similar assumptions as in Theorem~\ref{thmuniqnondeg}. More precisely, the authors assume instead of (H2) that $xg'(x)/ g(x)$ is non-increasing on $(\alpha,\beta)$, in dimensions $d\geq3$. We assume less on $(\alpha,\beta)$ but put the additional assumption that $I_\lambda$ does not vanish on $(0,\alpha)$. The function $I_\lambda$ in~\eqref{defI} appears already in~\cite{LeoSer-87,McLeod-93}. The method of proof in~\cite{SerTan-00} does not seem to provide the non-degeneracy of solutions. In the present article we clarify this important point by providing a self-contained proof in the spirit of McLeod~\cite{McLeod-93}, later in Section~\ref{sec:proof_thm_uniqueness}. Similar arguments were independently used in~\cite{KilOhPocVis-17,AdaShiWat-18}.

We conclude this section with a short discussion of the existence of solutions to~\eqref{nleq}. Let $g$ satisfy the condition (H1) of Theorem~\ref{thmuniqnondeg} on $[0,\beta]$ and extend it as a continuously differentiable function over $\R$ such that $g$ is odd and $g<0$ on $(\beta,\ii)$. Then we know from \cite{BerLio-83,BerGalKav-83} that, in dimension $d\geq2$, there exists at least one radial decreasing positive solution $u$ to~\eqref{nleq} which is $C^2$ and decays exponentially at infinity, if and only if 
\begin{equation}
G(\eta):=\int_{0}^\eta g(s)\,ds>0\ \text{for some $\eta>0$}. 
\label{condex3}
\end{equation}
That this condition is necessary follows from the Pohozaev identity
\begin{equation}\label{pohozaev}
\frac{d-2}{2d}\int_{\R^d}|\nabla u|^2\,dx=\int_{\R^d}G(u(x))\,dx
\end{equation}
which implies that $G(u(x))$ has to take positive values, it cannot be always negative (as it is for $x\to\ii$ since $G(s)\sim g'(0)s^2/2<0$ for $s\to0$). 

Finally, we recall that if $g\in C^{1+\varepsilon}$ for some $\varepsilon>0$ and since $g(0)=0$ and $g'(0)<0$, it follows from the moving plane method~\cite{GidNiNir-81} that all the positive solutions to \eqref{nleq} are radial decreasing about some point in $\R^d$.

\section{The double-power nonlinearity}\label{sec:double_power}

In this section we consider the nonlinear elliptic equation
\begin{equation}\label{powernl}
\boxed{-\Delta u=-u^{p}+u^{q}-\mu u }
\end{equation}
with $p>q>1$ and $\mu>0$ and $u^p:=|u|^{p-1}u$. This equation appears in a variety of practical situations, including density functional theory in quantum chemistry and condensed matter physics~\cite{LeBris-95,Ricaud-17}, Bose-Einstein condensates with three-body repulsive interactions~\cite{MerIsi-88}, heavy-ion collisons processes~\cite{Kartavenko-84} and nonclassical nucleation near spinodal in mesoscopic models of phase transitions~\cite{MurVan-04,CahHil-59,SaaHoh-92,CrosHoh-93}. The nonlinearity 
\begin{equation}
g_\mu(u)=-u^{p}+u^{q}-\mu u
\label{eq:def_g_mu}
\end{equation}
satisfies the condition (H1) of Theorem~\ref{thmuniqnondeg} hence, due to the Pohozaev identity~\eqref{pohozaev}, there exists a $\mu_*>0$ such that ~\eqref{powernl} admits no nontrivial solutions for $\mu\ge\mu_*$, whereas it always has at least one positive solution for $\mu\in(0,\mu_*)$. The value of $\mu_*$ is given by 
\begin{equation}\label{eq:mu_star}
\mu_*=\frac{2(p+1)^{\frac{q-1}{p-q}}(q-1)^{\frac{q-1}{p-q}}(p-q)}{(q+1)^{\frac{p-1}{p-q}}(p-1)^{\frac{p-1}{p-q}}}.
\end{equation}
For $\mu=\mu_*$ what is happening is that the primitive
$$G_\mu(u):=-\frac{|u|^{p+1}}{p+1}+\frac{|u|^{q+1}}{q+1}-\frac{\mu}2 |u|^2$$
becomes non-positive over the whole half line $\R_+$, with a double zero at
\begin{equation}
\beta_*:=\left(\frac{(q-1)(p+1)}{(q+1)(p-1)}\right)^{\frac{1}{p-q}}<1. 
\label{eq:def_beta_*}
\end{equation}
For all $0<\mu\leq\mu_*$, the function $g_\mu$ is seen to have two zeroes $0<\alpha_\mu<\beta_\mu$ such that 
$$\begin{cases}
\dps \lim_{\mu\to0^+}\alpha_\mu=0,  \\[0.2cm]
\dps \lim_{\mu\to0^+}\beta_\mu=1,
  \end{cases}\qquad\begin{cases}
\dps \lim_{\mu\to\mu_*^-}\alpha_\mu=\alpha_*\in(0,\beta_*), \\[0.2cm]
\dps \lim_{\mu\to\mu_*^-}\beta_\mu=\beta_*.
  \end{cases}$$
In addition $\mu\mapsto \beta_\mu$ is \emph{decreasing} over $(0,\mu_*)$.

\subsection{Branch parametrized by the Lagrange multiplier $\mu$}

The following result is a corollary of Theorem~\ref{thmuniqnondeg}. 

\begin{theorem}[Uniqueness and non-degeneracy]\label{thmpowernl} Let $d\geq2$, $p>q>1$ and $g_\mu$ as in~\eqref{eq:def_g_mu}. For all $\mu\in (0,\mu_*)$, the nonlinear equation~\eqref{powernl} has a \emph{unique positive solution} $u_\mu$ tending to $0$ at infinity, modulo space translations. It can be chosen \emph{radial-decreasing}. It is \emph{non-degenerate}:
\begin{equation}
\begin{cases}
\ker\big(-\Delta+pu_\mu^{p-1}-qu_\mu^{q-1}+\mu\big)={\rm span}\left\{\partial_{x_1}u_\mu,\ldots,\partial_{x_d}u_\mu\right\},\\
\ker\big(-\Delta+u_\mu^{p}-u_\mu^{q}+\mu\big)={\rm span}\{u_\mu\}.
\end{cases}
\label{eq:non_degenerate2}
\end{equation}
This solution satisfies
$$0<u_\mu(x)<\beta_\mu<1,\qquad\forall x\in\R^d$$ 
and $u_\mu(0)\to\beta_*$ when $\mu\nearrow\mu_*$.
\end{theorem}

Existence was proved in~\cite{BerLio-83, BerGalKav-83} and uniqueness in~\cite{SerTan-00, Jang-10}. The non-degeneracy of the solution follows from Theorem~\ref{thmuniqnondeg} as in~\cite{LewRot-15}. The cases $d\in\{2,3\},p=5,q=3$ and $d=3,p=7/3,q=5/3$ were handled in~\cite{KilOhPocVis-17,CarSpa-20_ppt} and~\cite{Ricaud-17,Ricaud-PhD}, respectively. Later in Theorem~\ref{thm:limit_mu_star} we will see that $u_\mu(x)\to\beta_*$ when $\mu\nearrow\mu_*$, for every $x\in\R^d$. The behavior of $u_\mu$ when $\mu\searrow0$ depends on the parameters $p$ and $q$, however, and will be given in Theorem~\ref{thm:limit_mu_0}.

\begin{proof} 
The existence part of the theorem is \cite[Example 2]{BerLio-83} for $d\ge 3$ and \cite{BerGalKav-83} for $d=2$. Moreover, $g_\mu(u)$ satisfies the hypotheses of \cite{GidNiNir-81} therefore all the positive solutions to \eqref{powernl} are radial decreasing about some point in $\R^d$. The function $g_\mu$ also satisfies hypothesis~\ref{hyp1} for some $0<\alpha_\mu<\beta_\mu$ and it is negative on $(\beta_\mu,\ii)$. Since $g$ is $C^\ii$ on $(0,\ii)$ and $u>0$, we deduce from regularity theory that $u$ is real-analytic on $\R^d$. We have $-\Delta u=g_\mu(u)<0$ on the open ball $\{u>\beta_\mu\}$, hence $u$ must be constant on this set, by the maximum principle. This definitely cannot happen for a real analytic function tending to 0 at infinity and therefore $u\leq\beta_\mu$ everywhere. The maximum of $u$ can also not be equal to $\beta_\mu$ since otherwise $u\equiv \beta_\mu$ which is the unique corresponding solution to~~\eqref{powernl}. We have therefore proved that all the positive solutions must satisfy $u<\beta_\mu$ and we are in position to apply Theorem~\ref{thmuniqnondeg}. It only remains to show that $g$ satisfies hypothesis~\ref{hyp2}. 

We first show that for all $x\in (0,\alpha_\mu)$ and for all $\lambda >1$, $I_\lambda(x)> 0$. 
To this end, we observe that for $x>0$ sufficiently small and $\lambda>1$,
\begin{equation*}
I_{\lambda}(x)=(1-\lambda)g_\mu'(0)x=(\lambda-1)\mu x>0
\end{equation*}
Next, by computing the second derivative of $g_\mu$, we remark that $g_\mu'$ is strictly increasing in $(0,x^*)$, attains his maximum at 
$$x^*=\left(\frac{q(q-1)}{p(p-1))}\right)^{\frac1{p-q}}$$ 
and is strictly decreasing in $(x^*,+\infty)$. Moreover, as a consequence of~\ref{hyp1}, $g'_\mu$ has to vanishes at least once. This implies that $g'_\mu$ vanishes exactly twice, namely at $x_1\in(0,\alpha_\mu)$ and at $x_2\in (\alpha_\mu,\beta_\mu)$. Hence $g_\mu'(x)<0$ for all $x\in (0,x_1)$ and $g_\mu'(x)>0$ for all $x\in (x_1,\alpha_\mu)\subset(x_1,x_2)$. As a consequence, since $g_\mu''(x)>0$ for all $x\in(0,x^*)$ and $x_1<x^*$, 
$$
I'_{\lambda}(x)=(1-\lambda)g_\mu'(x)+xg_\mu''(x)>0
$$
for all $x\in (0,x_1)$. This shows that $I_\lambda$ is strictly increasing in this interval and $I_\lambda(x)>0$ for all $x\in (0,x_1)$. Next, for all $x\in(x_1,\alpha_\mu)$, we have $g_\mu'(x)>0$ and $g_\mu(x)<0$, which implies $I_\lambda(x)>0$.

It remains to show that $I_\lambda$ vanishes exactly once in $(\alpha_\mu,\beta_\mu)$. It is clear that $I_\lambda$ has to vanish at least once since $I_\lambda(\alpha_\mu)=\alpha_\mu g_\mu'(\alpha_\mu)>0$ and $I_\lambda(\beta_\mu)=\beta_\mu g_\mu'(\beta_\mu)<0$. Moreover, it is proved in \cite{SerTan-00} that the function
$$
h(x)=\frac{xg_\mu'(x)}{g_\mu(x)}
$$ 
is decreasing on $(\alpha_\mu,\beta_\mu)$. This is enough to conclude that $I_\lambda$ has exactly one zero in $(\alpha_\mu,\beta_\mu)$. Hence $g$ satisfies hypothesis~\ref{hyp2} and this concludes the proof of the uniqueness and non-degeneracy in Theorem~\ref{thmpowernl}.

Since $\mu\mapsto \beta_\mu$ is decreasing and its limit at $\mu=0$ is $1$, we deduce that the family $(u_\mu)_\mu$ of solutions to \eqref{powernl} is uniformly bounded: $0<u_\mu<\beta_\mu<1$. If we denote by $\eta_\mu$ the first positive zero of $G_\mu$, then we also have $u_\mu(0)\geq \eta_\mu$, since $G_\mu(u_\mu(0))>0$ by~\eqref{pohozaev}. Since $\eta_\mu\to\beta_*$ when $\mu\nearrow\mu_*$, we obtain $u_\mu(0)\to\beta_*$ when $\mu\nearrow\mu_*$.
\end{proof}

\subsection{Behavior of the mass}
It is very important to understand how the mass of the solution $u_\mu$
\begin{equation}\label{defmass}
\boxed{ M(\mu):=\int_{\R^d} u_\mu(x)^2\;dx}
\end{equation}
varies with $\mu$. In the case of the usual focusing NLS equation with one power nonlinearity $q$ (which formally corresponds to $p=+\ii$ since $u<1$, at least when $q<2^*-1$), the mass is an explicit function of $\mu$ by scaling:
$$M_{\rm NLS}(\mu)=\mu^{\frac{4+d-dq}{2(q-1)}}\int_{\R^d}Q(x)^2\,dx$$
where $-\Delta Q-Q^q+Q=0$. There is no such simple relation for the double-power nonlinearity.

The importance of $M(\mu)$ is for instance seen in the Grillakis-Shatah-Strauss theory~\cite{Weinstein-85,ShaStr-85,GriShaStr-87,GriShaStr-90,BieGenRot-15} of stability for these solutions within the time-dependent Schr\"odinger equation. The latter says that the solution $u_\mu$ is \emph{orbitally stable} when $M'(\mu)>0$ and that it is \emph{unstable} when $M'(\mu)<0$. Therefore the intervals where $M$ is increasing furnish stable solutions whereas those where $M$ is decreasing correspond to unstable solutions. The Grillakis-Shatah-Strauss theory relies on another conserved quantity, the energy, which is discussed in the next section and for which the variations of $M$ also play a crucial role.

Note that the derivative can be expressed in terms of the linearized operator 
$$\boxed{\cL_\mu:=-\Delta-g'_\mu(u_\mu)=-\Delta+pu_\mu^{p-1}-qu_\mu^{q-1}+\mu}$$
by
\begin{equation}
M'(\mu)=2\Re\pscal{u_\mu,\frac{\partial}{\partial\mu} u_\mu}=-2\pscal{u_\mu,(\cL_\mu)^{-1}_{\rm rad}u_\mu}. 
 \label{eq:M'(mu)}
\end{equation}
Here $(\cL_\mu)^{-1}_{\rm rad}$ denotes the inverse of $\cL_\mu$ when restricted to the subspace of radial functions, which is well defined and bounded due to the non-degeneracy~\eqref{eq:non_degenerate2} of the solution.\footnote{The functions $\partial_{x_j}u_\mu$ spanning the kernel of $\cL_\mu$ are orthogonal to the radial sector, hence $0$ is not an eigenvalue of $(\cL_\mu)_{\rm rad}$. But then $0$ belongs to its resolvent set, since the essential spectrum starts at $\mu>0$.} This is why the non-degeneracy is crucial for understanding the variations of $M$. From the implicit function theorem, note that $M$ is a real-analytic function on $(0,\mu_*)$. 

\medskip

Our main goal is to understand the number of sign changes of $M'$, which tells us how many stable and unstable branches there are. Here is a soft version of a conjecture which we are going to refine later on. It states that there is at most one unstable branch. 

\begin{conjecture}[Number of unstable branches]
Let $d\geq2$ and $p>q>1$. The function $M'$ vanishes at most once on $(0,\mu_*)$.
\end{conjecture}

If true, this conjecture would have a number of interesting consequences, with regard to the stability of $u_\mu$ and the uniqueness of energy minimizers. 

We show in Theorem~\ref{thm:limit_mu_star} below that the stable branch is always present since $M'>0$ close to $\mu_*$. In order to make a more precise conjecture concerning the number of roots of $M'$ in terms of the exponents $p$ and $q$ and the dimension $d\geq2$, it is indeed useful to analyze the two regimes $\mu\to0$ and $\mu\to\mu_*$, where one can expect some simplification. This is the purpose of the next two subsections.

\subsubsection{The limit $\mu\searrow0$}
The following long statement about the limit $\mu\searrow0$ is an extension of results from~\cite{MorMur-14}, where the limit of $u_\mu$ was studied, but not that of $M$ and $M'$.

\begin{theorem}[Behavior when $\mu\searrow0$]\label{thm:limit_mu_0}
Let $d\geq2$ and $p>q>1$. 

\medskip

\noindent$\bullet$ \emph{\bf (Sub-critical case)} If $d=2$, or if $d\geq3$ and 
$$q<1+\frac{4}{d-2},$$ 
then the rescaled function
\begin{equation}
 \frac{1}{\mu^{\frac1{q-1}}}u_\mu\left(\frac{x}{\sqrt\mu}\right)
 \label{eq:rescaled_subcritical}
\end{equation}
converges strongly in $H^1(\R^d)\cap L^\ii(\R^d)$ in the limit $\mu\to0$ to the function $Q$ which is the unique positive radial-decreasing solution to the nonlinear Schr\"odinger (NLS) equation
\begin{equation}\label{limiteq0}
\Delta Q +Q^q-Q=0.
\end{equation}
We have
\begin{multline}
 M(\mu)=\mu^{\frac{4+d-dq}{2(q-1)}}\int_{\R^d}Q^2
 +\frac{2(p-1)+4+d-dq}{(p+1)(q-1)}\mu^{\frac{2(p-q)+4+d-dq}{2(q-1)}}\int_{\R^d}Q^{p+1}\\+o\left(\mu^{\frac{2(p-q)+4+d-dq}{2(q-1)}}\right)_{\mu\searrow0}
 \label{eq:limit_M_mu_0}
\end{multline}
and
\begin{multline}
M'(\mu)=\frac{4+d-dq}{2(q-1)}\mu^{\frac{4+d-dq}{2(q-1)}-1}\int_{\R^d}Q^2\\
 +\frac{(2(p-1)+4+d-dq)(2(p-q)+4+d-dq)}{2(p+1)(q-1)^2}\mu^{\frac{2(p-q)+4+d-dq}{2(q-1)}-1}\int_{\R^d}Q^{p+1}\\+o\left(\mu^{\frac{2(p-q)+4+d-dq}{2(q-1)}-1}\right)_{\mu\searrow0}.
\label{eq:limit_M_derivative_mu_0}
\end{multline}
In particular, $M$ is increasing for $q\leq 1+4/d$ and decreasing for $q>1+4/d$, in a neighborhood of the origin.

\bigskip

\noindent $\bullet$ \emph{\bf (Critical case)} If $d\geq3$ and 
$$q=1+\frac{4}{d-2},$$ 
then the rescaled function
\begin{equation}
\frac{1}{\eps_\mu^{\frac{d-2}{2}}}u_\mu\left(\frac{x}{\eps_\mu}\right) 
 \label{eq:rescaled_critical}
\end{equation}
converges strongly in $\dot{H}^1(\R^d)\cap L^\ii(\R^d)$ in the limit $\mu\to0$ to the Sobolev optimizer 
$$S(x)=\left(1+\frac{|x|^2}{d(d-2)}\right)^{-\frac{d-2}{2}},$$ 
which is also the unique positive radial-decreasing solution (up to dilations) to the Emden-Fowler equation
$\Delta S +S^q=0,$
where
\begin{equation}
\eps_\mu\sim c\begin{cases}
\mu^{\frac{1}{p-3}}&\text{if $d=3$,}\\
\left(\mu\log\mu^{-1}\right)^{\frac{1}{p-1}}&\text{if $d=4$,}\\
\mu^{\frac{q-1}{2(p-1)}}&\text{if $d\geq5$.}\\
\end{cases}
\label{eq:def_eps_mu}
\end{equation}
Furthermore, we have 
\begin{equation}
\lim_{\mu\searrow0} M(\mu) =\lim_{\mu\searrow0}-M'(\mu)=\ii.
 \label{eq:limit_M_mu_0_critical}
\end{equation}
In particular, $M$ is decreasing in a neighborhood of the origin.

\bigskip

\noindent $\bullet$ \emph{\bf (Super-critical case)} If $d\geq3$ and 
$$q>1+\frac{4}{d-2},$$ 
then $u_\mu$ converges strongly in $\dot{H}^1(\R^d)\cap L^\ii(\R^d)$ in the limit $\mu\to0$ to the unique positive radial-decreasing solution $u_0\in \dot{H}^1(\R^d)\cap L^{p+1}(\R^d)$ of the `zero-mass' double-power equation
$$-\Delta u_0=-u_0^p+u_0^q$$
decaying like $u_0(x)=O(|x|^{2-d})$ at infinity. We have the limits
\begin{equation}
\lim_{\mu\searrow0} M(\mu)=\int_{\R^d}u_0(x)^2\,dx\begin{cases}
=\ii&\text{if $d\in\{3,4\}$,}\\
<\ii&\text{if $d\geq5$}
\end{cases}
 \label{eq:limit_M_mu_0_supercritical}
\end{equation}
and 
\begin{equation}
\lim_{\mu\searrow0} M'(\mu)=\begin{cases}
-\ii&\text{if $d\in\{3,4,5,6\}$,}\\
  M'(0)\in\R&\text{if $d\geq7$.}
\end{cases}
 \label{eq:limit_derivative_M_mu_0_supercritical}
\end{equation}
In particular, $M$ is decreasing in a neighborhood of the origin when $d\in\{3,...,6\}$. 
In dimensions $d\geq7$, we have $M'(0)<0$ under the additional condition
\begin{equation}
1+\frac{4}{d-2}<q<p<1+\frac{4}{d-2}+\frac{32}{d(d-2)\big((d-2)q-d-2\big)}.
\label{eq:ugly_condition_on_p}
\end{equation}
\end{theorem}

The convergence properties of $u_\mu$ are taken from~\cite{MorMur-14} in all cases. Only the behavior of $M$ and $M'$ is new. The corresponding proof is given below in Section~\ref{sec:proof_mu_0}.  

The condition~\eqref{eq:ugly_condition_on_p} in the super-critical case is not at all expected to be optimal and it is only provided as an illustration. This condition requires that 
$$1+\frac{4}{d-2}<q<1+\frac{4}{d-2}+\frac{4\sqrt2}{\sqrt{d}(d-2)}$$
and that $p$ satisfies the right inequality in~\eqref{eq:ugly_condition_on_p}. In particular, $p$ can be arbitrarily large when $q$ approaches the critical exponent. Although  we are able to prove that $M'$ admits a finite limit when $\mu\to0$ in dimensions $d\geq7$, we cannot determine its sign in the whole range of parameter. Numerical simulations provided below in Section~\ref{sec:conjecture} seem to indicate that $M'(0)$ can be positive. The limit $\mu\to0$ for $M'(\mu)$ is quite delicate in the super-critical case, since the limiting linearized operator 
$$\cL_0=-\Delta+p(u_0)^{p-1}-q(u_0)^{q-1}$$ 
has no gap at the origin. Its essential spectrum starts at $0$. Nevertheless, we show in Appendix~\ref{app:u_0} that $u_0$ is still \emph{non-degenerate} in the sense that $\ker\left(\cL_0\right)={\rm span}\left\{\partial_{x_1}u_0,...,\partial_{x_d}u_0\right\}$.
This allows us to define $(\cL_0)_{\rm rad}^{-1}$ by the functional calculus and to prove that, as expected,
$$M'(0)=-2\pscal{u_0,(\cL_0)_{\rm rad}^{-1}u_0},$$
where the right side is interpreted in the sense of quadratic forms. In dimensions $d\geq5$ there are no resonances and $(\cL_0)_{\rm rad}^{-1}$ essentially behaves like $(-\Delta)^{-1}_{\rm rad}$ at low momenta~\cite{Jensen-80}. Since $\pscal{u_0,(-\Delta)^{-1}u_0}$ is finite only in dimensions $d\geq7$ due to the slow decay of $u_0$ at infinity, $M'(0)$ is only finite in those dimensions. In dimensions $d\in\{7,8\}$ it is the second derivative $M''(\mu)$ which diverges to $+\ii$ when $\mu\to0$ (see Remark~\ref{rmk:higher_derivatives}) but this does not tell us anything about the variations of $M$. 

\subsubsection{The limit $\mu\nearrow\mu_*$}
Next we study the behavior of the branch of solutions in the other limit $\mu\nearrow\mu_*$.

\begin{theorem}[Behavior when $\mu\nearrow\mu_*$]\label{thm:limit_mu_star}
Let $d\geq2$ and $p>q>1$. Let $\mu_*$ and $\beta_*$ be the two critical constants defined in~\eqref{eq:mu_star} and~\eqref{eq:def_beta_*}, respectively. Then we have
\begin{equation}
\boxed{\lim_{\mu\nearrow\mu_*}(\mu_*-\mu)^dM(\mu)=\lim_{\mu\nearrow\mu_*}\frac{(\mu_*-\mu)^{d+1}}{d}M'(\mu)=\Lambda}
\label{eq:limit_M_mu_infty}
\end{equation}
where
\begin{equation}
\Lambda:=2^{\frac{3d}2}\frac{|\bS^{d-1}|}{d}(\beta_*)^{2(1-d)}(d-1)^{d}\left(\int_{0}^{\beta_*}|G_{\mu_*}(s)|^{\frac12}\,ds\right)^d.
\label{eq:def_Lambda}
\end{equation}
Let $\gamma\in(0,\beta_*)$ be any constant and call $R_\mu$ the unique radius such that $u_\mu(R_\mu)=\gamma$. Then we have 
\begin{equation}
R_\mu=\frac{\rho}{\mu_*-\mu}+o\left(\frac1{\mu_*-\mu}\right),\qquad \rho=\frac{2\sqrt2(d-1)}{\beta_*^2}\int_0^{\beta_*} \sqrt{|G_{\mu_*}(s)|}\,ds, 
 \label{eq:formula_R_mu}
\end{equation}
and the uniform convergence
\begin{equation}
\lim_{\mu\to\mu_*}\norm{u_\mu-U_*(|x|-R_\mu)}_{L^\ii(\R^d)}=0,
\label{eq:limit_mu_mu_*}
\end{equation}
where $U_*$ is the unique solution to the one-dimensional limiting problem 
\begin{equation}
 \begin{cases}
U_*''+g_{\mu_*}(U_*)=0&\text{on $\R$}\\
U_*(-\ii)=\beta_*\\
U_*(+\ii)=0\\
U_*(0)=\gamma\in(0,\beta_*.)
  \end{cases}
 \label{eq:U}
\end{equation}
\end{theorem}

The proof will be provided later in Section~\ref{sec:proof_mu_star}. What the result says is that $u_\mu$ ressemble a \emph{radial translation} of the one-dimensional solution $U_*$, which links the two unstable stationary solutions $\beta_*$ and $0$ of the underlying Hamiltonian system. Since $U_*$ tends to $\beta_*$ at $-\ii$, we see that $u_\mu(r)$ tends to $\beta_*$ for every fixed $r$, as we claimed earlier, and this is why the mass diverges like 
$$M(\mu)\underset{\mu\to\mu_*}\sim(R_\mu)^d(\beta_*)^2\frac{|\bS^{d-1}|}{d}.$$ 
Plugging the asymptotics of $R_\mu$ from~\eqref{eq:formula_R_mu} then provides~\eqref{eq:limit_M_mu_infty}.
Stronger convergence properties of $u_\mu$ will be given in the proof. For instance we have for the derivatives
$$\lim_{\mu\to\mu_*}\norm{u'_\mu-U_*'(|x|-R_\mu)}_{L^\ii(\R^d)}=\lim_{\mu\to\mu_*}\int_0^\ii\left|u'_\mu(r)-U'_*(r-R_\mu)\right|^p\,dr=0$$
for all $1\leq p<\ii$. On the other hand we only have 
$$\|u'_\mu-U_*'(|x|-R_\mu)\|_{L^p(\R^d)}=o(R_\mu^{\frac{d-1}{2}})$$ 
due to the volume factor $r^{d-1}{\rm d}r$. 

Upper and lower bounds on $M(\mu)$ in terms of $(\mu_*-\mu)^{-d}$ were derived in~\cite{KilOhPocVis-17} in the case $d=3$, $p=5$ and $q=3$ but the exact limit~\eqref{eq:limit_M_mu_infty} is new, to our knowledge. Particular sequences $\mu_n\to\mu_*$ have been studied in~\cite{AkaKikYam-18}. That the solution $u_\mu$ tends to a constant in the limit of a large mass is a `saturation phenomenon' which plays an important role in Physics, for instance for infinite nuclear matter~\cite{MerIsi-88}. 

Theorem~\ref{thm:limit_mu_star} implies that $M$ is always increasing close to $\mu_*$, hence in this region we obtain an orbitally stable branch for the Schr\"odinger flow, for every $p>q>1$. 

\begin{remark}[A general result]
Our proof of Theorem~\ref{thm:limit_mu_star} is general and works the same for a function in the form $g_\mu(u)=g_0(u)-\mu u$ with
\begin{itemize}
\item $g_0\in C^1([0,\ii))\cap C^2(0,\ii)$ with $g_0(0)=g_0'(0)=0$ and $g_0(s)\to-\ii$ when $s\to+\ii$;
 \item $g_\mu$ has exactly two roots $0<\alpha_\mu<\beta_\mu$ on $(0,\ii)$ with $g_\mu'(\alpha_\mu)>0$ and $g_\mu'(\beta_\mu)<0$ for all $\mu\in(0,\mu_*]$ where $\mu_*$ is the first $\mu$ so that $G_\mu(r)=\int_0^rg_\mu(s)\,ds\leq0$ for all $r\geq0$;
 \item $\Delta u+g_\mu(u)=0$ has a unique non-degenerate radial positive solution for every $\mu\in(0,\mu_*)$ (for instance $g_\mu$ satisfies \textnormal{(H2)} in Theorem~\ref{thmuniqnondeg} for all $\mu\in(0,\mu_*)$). 
\end{itemize}
\end{remark}

\subsubsection{Main conjecture and numerical illustration}\label{sec:conjecture}

Theorems~\ref{thm:limit_mu_0} and~\ref{thm:limit_mu_star} and the fact that $M$ is a smooth function on $(0,\mu_*)$ imply some properties of solutions to the equation $M(\mu)=\lambda$, whenever $\lambda$ is either small or large. Those are summarized in the following

\begin{corollary}[Number of solutions to $M(\mu)=\lambda$]\label{cor:equation_M_lambda}
Let $d\geq2$ and $p>q>1$. The equation 
$$M(\mu)=\lambda$$

\medskip

\noindent$\bullet$ admits a \emph{unique solution} $\mu$ for $\lambda$ small enough when $1<q<1+\frac{4}{d}$, and it is stable, $M'(\mu)>0$;

\medskip

\noindent$\bullet$ admits a \emph{unique solution} $\mu$ for $\lambda$ large enough when 
$$\begin{cases}
\text{$1<q\leq1+\frac4d$,}\\
\text{$q>1+\frac{4}{d-2}$ and $d\geq5$,}   
  \end{cases}
$$
and it is stable, $M'(\mu)>0$;

\medskip

\noindent $\bullet$ admits \emph{exactly two} solutions $\mu_1<\mu_2$ for $\lambda$ large enough when 
$$\begin{cases}
\text{$q>1+\frac{4}{d}$ and $d\in\{2,3,4\}$,}\\
\text{$1+\frac4d<q\leq 1+\frac{4}{d-2}$ and $d\geq5$,}\\
\end{cases}
$$
which are respectively unstable and stable: $M'(\mu_1)<0$, $M'(\mu_2)>0$.
\end{corollary}

Now that we have determined the exact behavior of $M$ at the two end points of its interval of definition, it seems natural to expect that the following holds true. 

\begin{conjecture}[Behavior of $M$]\label{conjecture:M}
Let $d\geq2$ and $p>q>1$. Then $M'$ is either positive on $(0,\mu_*)$, or vanishes at a unique $\mu_c\in(0,\mu_*)$ with
\begin{equation}
 M'\begin{cases}
<0&\text{on $(0,\mu_c)$,}\\
>0&\text{on $(\mu_c,\mu_*)$.}
\end{cases}
 \label{eq:conjecture_M'}
\end{equation}
More precisely:
\begin{itemize}
 \item[$\bullet$] If $q\leq 1+4/d$, then $M'>0$ on $(0,\mu_*)$.
 \item[$\bullet$] If $d\in\{2,...,6\}$ and $q>1+4/d$, or if $d\geq7$ and $1+4/d<q\leq 1+4/(d-2)$, then $M'$ vanishes exactly once. 
 \item[$\bullet$] If $d\geq7$ and $q>1+4/(d-2)$, there exists a $p_c(q)\geq q$ such that $M'$ vanishes once for $q<p< p_c(q)$ and does not vanish for $p> p_c(q)$.
\end{itemize}
\end{conjecture}

The property~\eqref{eq:conjecture_M'} is an immediate consequence of Theorems~\ref{thm:limit_mu_0} and~\ref{thm:limit_mu_star} whenever $M'$ vanishes only once. The conjecture was put forward in~\cite{KilOhPocVis-17,CarSpa-20_ppt} for the quintic-cubic NLS equation ($p=5,q=3$) in dimensions $d\in\{2,3\}$, and in~\cite{Ricaud-17} for $d=3$, $p=7/3,q=5/3$. These cases have been confirmed by numerical simulations~\cite{Anderson-71,MerIsi-88,KilOhPocVis-17,Ricaud-17}. 

In Figures~\ref{fig:numerics_2D}--\ref{fig:numerics_5D_7D} we provide a selection of numerical simulations of the function $M$ in dimensions $d\in\{2,3,5,7\}$ which seem to confirm the conjecture. Although we have run many more simulations and could never disprove the conjecture, we have however not investigated all the possible powers and dimensions in a systematical way.

\begin{figure}[h]
\begin{tabular}{ccc}
\includegraphics[width=4cm]{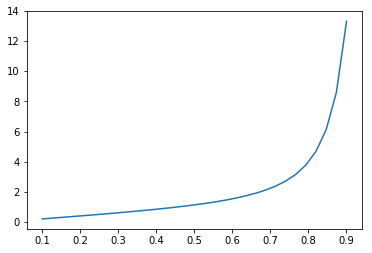}&\includegraphics[width=4cm]{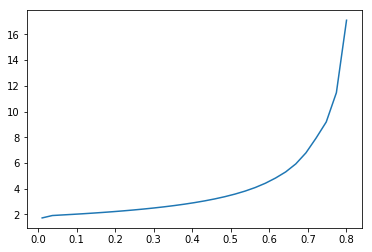}&\includegraphics[width=4cm]{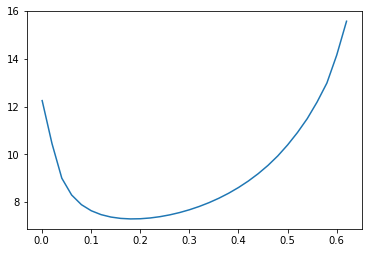}\\
$p=5$, $q=2$&$p=5$, $q=3$&$p=5$, $q=4$\\
\end{tabular}
\caption{Function $\mu/\mu_*\in[0,1)\mapsto |\bS^{d-1}|^{-1}M(\mu)$ in dimension $d=2$.\label{fig:numerics_2D}}

\bigskip\bigskip

\begin{tabular}{ccc}
\includegraphics[width=4cm]{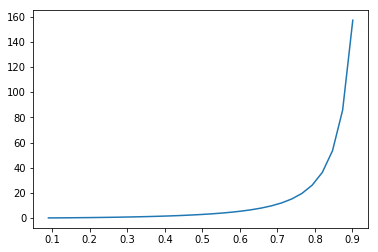}&\includegraphics[width=4cm]{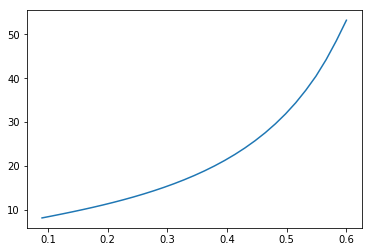}&\includegraphics[width=4cm]{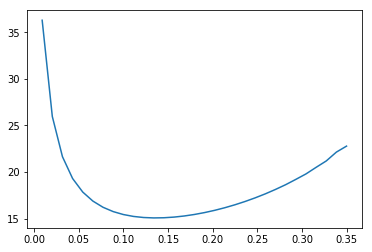}\\
$p=7/3$, $q=5/3$&$p=3$, $q=7/3$&$p=5$, $q=3$\\
\end{tabular}
\caption{Same function in dimension $d=3$.\label{fig:numerics_3D}}

\bigskip\bigskip

\begin{tabular}{ccc}
\includegraphics[width=4cm]{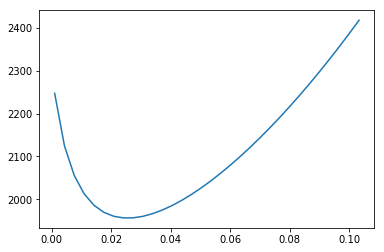}&\includegraphics[width=4cm]{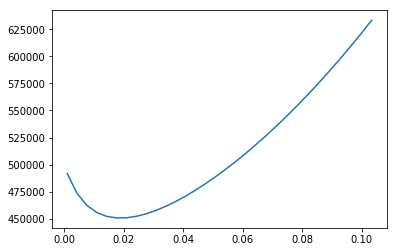}&\includegraphics[width=4cm]{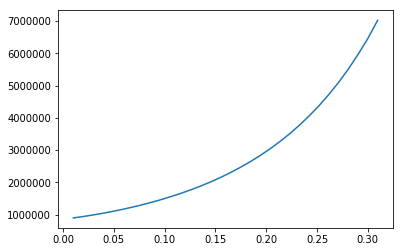}\\
$d=5$, $p=5$, $q=3$&$d=7$, $p=5/2$, $q=2$&$d=7$, $p=5$, $q=3$\\
\end{tabular}
\caption{Same function in dimensions $d=5$ and $d=7$. The second case $p=5/2,q=2$ is covered by~\eqref{eq:ugly_condition_on_p} whereas the third is not. \label{fig:numerics_5D_7D}}
\end{figure}

\subsection{The double-power energy functional}\label{sec:double_power_energy}
In this paper the larger power $p$ is defocusing and always controls the smaller focusing nonlinearity of exponent $q$. In this situation the double-power NLS equation~\eqref{powernl} has a natural variational interpretation in the whole possible range of powers, which we discuss in this section. We introduce the energy functional
$$\cE(u)=\frac12\int_{\R^d}|\nabla u(x)|^2\,dx+\frac1{p+1}\int_{\R^d}|u(x)|^{p+1}\,dx-\frac1{q+1}\int_{\R^d}|u(x)|^{q+1}\,dx$$
and the corresponding minimization problem
\begin{equation}
\boxed{I(\lambda):=\inf_{\substack{u\in H^1(\R^d)\cap L^{p+1}(\R^d)\\ \int_{\R^d}|u|^2=\lambda}}\cE(u)}
\label{eq:I_lambda}
\end{equation}
at fixed mass $\lambda\geq0$. This problem is well posed for all $p>q>1$ because we can write
$$\cE(u)= \frac12\int_{\R^d}|\nabla u(x)|^2\,dx-\int_{\R^d}G_{\mu_*}\big(u(x)\big)\,dx-\frac{\mu_*\lambda}{2}\geq -\frac{\mu_*\lambda}{2}.$$
Recall that $\mu_*$ in~\eqref{eq:mu_star} is precisely the lowest $\mu$ for which $G_\mu\leq0$ on $\R_+$. The minimization problem~\eqref{eq:I_lambda} appears naturally in applications, for instance in condensed matter physics for $d=3$, $p=7/3$ and $q=5/3$ where it can be obtained from the Thomas-Fermi-von Weis\"acker-Dirac functional of atoms, molecules and solids~\cite{ChaCohHan-01,Lieb-81b,BenBreLie-81,Lions-87,Lions-88,LeBris-93b}, in a certain limit of a large Dirac term~\cite{Ricaud-17,GonLewNaz-20_ppt}.

The existence of minimizers follows from rather standard methods of nonlinear analysis, as stated in the following

\begin{theorem}[Existence of minimizers for $I(\lambda)$]\label{thm:prop_I_lambda}
Let $d\geq2$ and $p>q>1$. 
The function $\lambda\mapsto I(\lambda)$ is concave non-increasing over $[0,\ii)$. It satisfies
\begin{itemize}
\item[$\bullet$] $I(\lambda)=0$ for all $0\leq \lambda\leq \lambda_c$,
\item[$\bullet$] $\lambda\mapsto I(\lambda)$ is negative and strictly decreasing on $(\lambda_c,\ii)$,
 \end{itemize}
 where
 $$\lambda_c\begin{cases}
 =0&\text{if $q<1+4/d$,}\\
=\int_{\R^d}Q^2&\text{if $q=1+4/d$,}\\
 \in(0,\ii)&\text{if $q>1+4/d$,}
 \end{cases}$$
with $Q$ the same NLS function as in Theorem~\ref{thm:limit_mu_0}.
The problem $I(\lambda)$ admits at least one positive radial-decreasing minimizer $u$ for every 
$$\lambda\begin{cases}
\geq \lambda_c&\text{if $q\neq1+4/d$,}\\
> \lambda_c&\text{if $q=1+4/d$.}
         \end{cases}$$
Any minimizer $u$ solves the Euler-Lagrange equation~\eqref{powernl} for some $\mu\in(0,\mu_*)$, hence must be equal to $u_\mu$. 
The infimum is not attained for $\lambda<\lambda_c$ or for $\lambda=\lambda_c$ and $q=1+4/d$. 
\end{theorem}

In the proof, provided later in Section~\ref{sec:prop_I_lambda}  we give a characterization of $\lambda_c$ in terms of optimizers of the Gagliardo-Nirenberg-type inequality
\begin{equation}
\norm{u}_{L^{q+1}(\R^d)}^{q+1}\leq C_{p,q,d}\norm{u}_{L^2(\R^d)}^{q-1-\theta(p-1)}\norm{\nabla u}_{L^2(\R^d)}^{2(1-\theta)}\norm{u}_{L^{p+1}(\R^d)}^{\theta(p+1)}
\label{eq:Gagliardo-Nirenberg-type}
\end{equation}
when $q\geq1+4/d$, with 
$$\theta=\frac{q-1-\frac{4}d}{p-1-\frac{4}d}\in[0,1).$$
A similar property was used in~\cite{KilOhPocVis-17,CarSpa-20_ppt}. At $q=1+4/d$ we have $\theta=0$ and obtain the usual Gagliardo-Nirenberg inequality, of which $Q$ is the unique optimizer. 

A very natural question is to ask whether minimizers of $I(\lambda)$ are \emph{unique}, up to space translations and multiplication by a phase factor. This does not follow from the uniqueness of $u_\mu$ at fixed $\mu$ because the minimizers could have different multipliers $\mu$'s. The concavity of $I$ implies that it is differentiable except for countably many values of $\lambda$. When the derivative exists and $\lambda>\lambda_c$, it can be seen that the minimizer is unique and given by $u_\mu$ with $\mu=-2I'(\lambda)$. Details will be provided later in Theorem~\ref{thm:uniqueness_I_lambda} where we actually show that the derivative can only have finitely many jumps in $(\lambda_c,\ii)$.  

Another natural question is to ask whether one solution $u_\mu$ could be a candidate for the minimization problem $I(\lambda)$ with $\lambda=M(\mu)$. From the non-degeneracy of $u_\mu$, the answer (see, e.g.~\cite[App.~E]{Weinstein-85}) is that when $M'(\mu)>0$ the corresponding solution $u_\mu$ is a \emph{strict local minimum} of $\cE$ at fixed mass $\lambda=M(\mu)$, whereas when $M'(\mu)<0$, the solution $u_\mu$ is a saddle point. In particular, there must always hold $M'(\mu)\geq0$ for a minimizer $u_\mu$ of $I(\lambda)$. 

From this discussion, we see that the following would immediately follow from Conjecture~\ref{conjecture:M}.

\begin{conjecture}[Uniqueness of minimizers]\label{conjecture:I}
Let $d\geq2$ and $p>q>1$. Then $I(\lambda)$ admits a \emph{unique minimizer} for all $\lambda\geq\lambda_c$ (resp. $\lambda>\lambda_c$ if $q=1+4/d$). 
\end{conjecture}

Although we are not able to prove this conjecture, our previous analysis implies the following uniqueness result. 

\begin{theorem}[Partial uniqueness of minimizers]\label{thm:uniqueness_I_lambda}
Let $d\geq2$ and $p>q>1$. Then $I(\lambda)$ admits a \emph{unique positive radial minimizer} when 
\begin{itemize}
 \item $\lambda$ is large enough;
 \item $q<1+4/d$ and $\lambda\in[0,\eps)$,
 \item $q\geq 1+4/d$ and $\lambda\in(\lambda_c,\lambda_c+\eps)$
\end{itemize}
for some $\eps>0$ small enough. 
In fact, $I(\lambda)$ has a unique positive radial minimizer for all $\lambda\in[\lambda_c,\ii)$ (resp. $\lambda\in(\lambda_c,\ii)$ when $q=1+4/d$), except possibly at finitely many points in $[\lambda_c,\ii)$. At those values, the number of positive radial minimizers is also finite. 
For any $\lambda\in(\lambda_c,\ii)$ we have 
$$I'(\lambda^-)=-\frac12\min\big\{\mu\ :\ \cE(u_\mu)=I(\lambda),\ M(\mu)=\lambda\big\},$$
and 
$$I'(\lambda^+)=-\frac12\max\big\{\mu\ :\ \cE(u_\mu)=I(\lambda),\ M(\mu)=\lambda\big\}.$$
\end{theorem}

In order to explain the proof of Theorem~\ref{thm:uniqueness_I_lambda}, it is useful to introduce the energy
$$E(\mu):=\cE(u_\mu),\qquad \mu\in(0,\mu_*)$$
of our branch of solutions $u_\mu$. Note that 
\begin{equation}
E'(\mu)=-\frac{\mu}2 M'(\mu),
\label{eq:relation_E_M}
\end{equation}
that is, the variations of $E$ are exactly opposite to those of $M$. The following is a simple consequence of Theorems~\ref{thm:limit_mu_0} and~\ref{thm:limit_mu_star} together with~\eqref{eq:relation_E_M}.

\begin{corollary}[$E(\mu)$ at $0$ and $\mu_*$]\label{cor:Energy}
Let $d\geq2$ and  $p>q>1$. 

\medskip

\noindent$\bullet$ When $\mu\searrow0$, we have 
$$\lim_{\mu\to0^+}E(\mu)=\begin{cases}
\cE(u_0)=\frac{1}{d}\int_{\R^d}|\nabla u_0|^2&\text{if $d\geq3$ and $q> 1+\frac4{d-2}$,}\\
\frac{1}{d}\int_{\R^d}|\nabla S|^2 &\text{if $d\geq3$ and $q= 1+\frac4{d-2}$,}\\
0&\text{otherwise.}\\
\end{cases}$$
Moreover
$$E(\mu)\begin{cases}
<0&\text{for $q\leq  1+4/d$,}\\
>0&\text{for $q> 1+4/d$,}
        \end{cases}$$
for $\mu$ in a neighborhood of the origin. 

\medskip

\noindent$\bullet$ When $\mu\nearrow\mu_*$, we have 
$$E(\mu)\underset{\mu\to\mu_*}\sim-\frac{\mu_*\Lambda}{2(\mu_*-\mu)^d}$$
where $\Lambda$ is the same constant as in Theorem~\ref{thm:limit_mu_star}.

\medskip

\noindent$\bullet$ $\mu\mapsto E(\mu)$ is real-analytic on $(0,\mu_*)$ and the equation $E(\mu)=e$ always has finitely many solutions for any $e\in(-\ii,\max E]$.
\end{corollary}

In the case $e=E(0)>0$ when $q>1+4/(d-2)$ and $d\geq7$, one has to use Remark~\ref{rmk:higher_derivatives} which says that one derivative $M^{(k)}$ always diverges in the limit $\mu\to0$, for $k$ large enough depending on the dimension $d$. This implies that $E$ cannot take the value $E(0)$ infinitely many times in a neighborhood of the origin. 

We see that Conjecture~\ref{conjecture:I} would follow if we could prove that 
\begin{itemize}
 \item $E$ is decreasing for $q\leq 1+4/d$;
 \item $E$ has a unique positive zero and is decreasing on the right side of this point, for $q>1+4/d$. 
\end{itemize}
Note that when $q>1+4/d$, Conjecture~\ref{conjecture:I} is really weaker than Conjecture~\ref{conjecture:M} on the mass $M(\mu)$, since the places where $E(\mu)>0$ do not matter for the minimization problem $I(\lambda)$. 

We conclude the section with the 

\begin{proof}[Proof of Theorem~\ref{thm:uniqueness_I_lambda}]
If $\lambda$ is large or small, the statement follows immediately from Corollary~\ref{cor:equation_M_lambda} (and the fact that $M'(\mu)\geq0$ at a minimizer $u_\mu$ for $I(\lambda)$ in case there are two solutions to the equation $M(\mu)=\lambda$). 

We now discuss the more complicated case $q=1+4/d$. We know from Theorems~\ref{thm:limit_mu_0} and~\ref{thm:prop_I_lambda} that $\lambda_c=\int_{\R^d}Q^2=M(0)$. We claim that minimizers for $\lambda$ close to $\lambda_c$ necessarily have $\mu$ small enough, so that the conclusion follows from the monotonicity of $M$ close to the origin, by Theorem~\ref{thm:limit_mu_0}. To see this, assume by contradiction that there exists a sequence $\mu_n\nrightarrow0$ such that $\cE(u_{\mu_n})=I(\lambda_n)\nearrow0$ and $\lambda_n=M(\mu_n)\searrow \lambda_c$. Since $E$ diverges to $-\ii$ at $\mu_*$, we can assume after extracting a subsequence that $\mu_n\to\mu\in(0,\mu_*)$ and then obtain $E(\mu)=0$ with $M(\mu)=\lambda_c$. But this cannot happen because
\begin{align*}
0=E(\mu)&=\frac12\int_{\R^d}|\nabla u_\mu|^2+\frac1{p+1}\int_{\R^d}u_\mu^{p+1}-\frac1{q+1}\int_{\R^d}u_\mu^{q+1} \\
&\geq \frac1{p+1}\int_{\R^d}u_\mu^{p+1}>0,
\end{align*}
where we used the Gagliardo-Nirenberg inequality~\eqref{eq:Gagliardo-Nirenberg-type} which, in the case $q=1+4/d$, takes the simple form
$$\frac1{q+1}\int_{\R^d}u^{q+1}\leq \frac{\norm{u}^{\frac4d}_{L^2(\R^d)}}{2(\lambda_c)^{\frac2d}}\int_{\R^d}|\nabla u|^2.$$
As a conclusion, all the minimizers for $I(\lambda)$ must have a small Lagrange multiplier $\mu$ when $\lambda$ is close to $\lambda_c$, and the result follows when $q=1+4/d$. 

For every $\lambda>\lambda_c$, the number of $\mu$'s such that $E(\mu)=I(\lambda)<0$ is finite by Corollary~\ref{cor:Energy}. The same holds at $\lambda_c$ when $q\geq1+4/d$. Hence $I(\lambda)$ always admits finitely many positive radial minimizers. 

It remains to prove that there can be at most finitely many $\lambda$'s for which uniqueness does not hold. Let us denote by $J_k$ all the disjoint closed intervals on which $M$ is increasing. By real-analyticity and the behavior of $M$ close to $0$ and $\mu_*$ from Theorems~\ref{thm:limit_mu_0}  and~\ref{thm:limit_mu_star}, there are only finitely many such intervals.\footnote{When $q>1+4/(d-2)$ and $d\geq7$ we need to use again Remark~\ref{rmk:higher_derivatives} to ensure that $M'$ cannot change sign infinitely many times close to the origin. In any case we have $E(0)>0$ by Corollary~\ref{cor:Energy} and this region does not play any role in the rest of the argument.} Note that in each interval $J_k$ the derivative $M'$ can still vanish, but if it does so this can only happen at finitely many points by real-analyticity. In the rest of the argument we label the intervals $J_k$ in increasing order, that is, such that $\mu<\mu'$ for all $\mu\in J_k$ and all $\mu'\in J_{k+1}$.  If $q\leq1+4/d$ then $J_1=[0,\mu_1]$ and $M'$ is positive close to the origin. On each $J_k$ we have a well defined continuous inverse $\mu_k:=M_{|J_k}^{-1}$ and the corresponding continuous energy $I_k(\lambda)=E(\mu_k(\lambda))$, $\lambda\in J'_k:=M(J_k)$. Since $M$ and $-E$ are increasing over $J_k$, then $\mu_k$ is increasing and $I_k$ is decreasing over $J'_k$. The intervals $J'_k$ cover the whole interval $[\min M,\ii)\supset [\lambda_c,\ii)$ and it is clear from the existence of minimizers in Theorem~\ref{thm:prop_I_lambda} that 
\begin{equation}
 I(\lambda)=\min\big\{0, I_k(\lambda)\big\}.
 \label{eq:I_min_E_k}
\end{equation}
The function $\lambda\mapsto I_k(\lambda)$ is real-analytic on the open subset $M(\{M'>0\}\cap J_k)$ of $J'_k$ and a calculation shows that
\begin{equation}
I_k'(\lambda)=-\frac{\mu_k(\lambda)}{2}=-\frac{M_{|J_k}^{-1}(\lambda)}{2}
\label{eq:derivative_E_k_lambda}
\end{equation}
on this set, that is, the derivative of the energy with respect to the mass constraint is proportional to the Lagrange multiplier. Using~\eqref{eq:derivative_E_k_lambda} or~\eqref{eq:relation_E_M} we see that $I_k$ is actually $C^1$ over the whole interval $J'_k$, with the relation~\eqref{eq:derivative_E_k_lambda} (but it need not be smoother in general). Note that since $\mu_k$ is increasing we deduce from~\eqref{eq:derivative_E_k_lambda} that $I_k$ is concave over $J'_k$. A last important remark is that due to~\eqref{eq:derivative_E_k_lambda} and the fact that the $J_k$ are disjoint ordered intervals, we see that 
$$\max I_{k+1}'<\min I'_k<0.$$
The slopes of $I_{k+1}$ are always more negative than the slopes of $I_k$ at any possible point. This property implies that two functions $I_k$ and $I_{k'}$ can cross at most once, with $I_{k'}$ strictly below $I_k$ on the right of the crossing point for $k'>k$, and conversely on the left. Thus there must be at most finitely many crossing points between the functions $I_k$ on $[\lambda_c,\ii)$. The function $I$ being the minimum of all these functions $I_k$ (and the constant function 0), we deduce that $I$ is always equal to exactly one of the $I_k$, except at finitely many points. This proves the statement that there is always only one minimizer, except at finitely many possible values of $\lambda$, where the $I_k$ cross and realize the minimum in~\eqref{eq:I_min_E_k}. At any such crossing point $\lambda_0$, we have $I(\lambda)=I_k(\lambda)$ for $\lambda\in(\lambda_0-\eps,\lambda_0)$ and $I(\lambda)=I_{k'}(\lambda)$ for $\lambda\in(\lambda_0,\lambda_0+\eps)$, where $k$ corresponds to the lowest possible multipliers, that is the interval $J_k$ which is the furthest to the left and $k'$ corresponds to the largest possible multipliers. 
This concludes the proof of Theorem~\ref{thm:uniqueness_I_lambda}.
\end{proof}

Note that $\lambda_c$ is exactly known when $q\leq 1+4/d$, but it is in principle not explicit for $q>1+4/d$. In case Conjecture~\ref{conjecture:M} holds true, then $\lambda_c=M(\mu_c')$ where $\mu_c'$ is the unique positive root of $E$, which is necessarily on the right of $\mu_c$, the unique point at which $M'(\mu_c)=0$. 

\medskip

The rest of the paper is devoted to the proof of our other main results. 

\section{Proof of Theorem~\ref{thm:limit_mu_0} on the limit $\mu\to0$}\label{sec:proof_mu_0}

The convergence of $u_\mu$ in all cases is proved in~\cite{MorMur-14}. We only discuss here the behavior of $M(\mu)$ and its derivative, which was not studied for all cases in~\cite{MorMur-14}. Recall that $M'(\mu)$ is given by~\eqref{eq:M'(mu)}.

\subsection{$M$ and $M'$ in the sub-critical case} When $q<1+4/(d-2)$ the rescaled function $\tilde u_\mu$ in~\eqref{eq:rescaled_subcritical} solves 
\begin{equation}\label{rescaledeq}
-\Delta \tilde u_\mu=-\mu^{\frac{p-q}{q-1}}\tilde u_\mu^p +\tilde u_\mu^q-\tilde u_\mu
\end{equation}
and it converges to the NLS solution $Q$. More precisely, the implicit theorem gives
\begin{equation}
\norm{\tilde u_\mu-Q+\mu^{\frac{p-q}{q-1}}(\cL_Q)_{\rm rad}^{-1}Q^p}_{H^1(\R^d)}=o\left(\mu^{\frac{p-q}{q-1}}\right)
\label{eq:expansion_tilde_u_mu}
\end{equation}
where $\cL_Q:=-\Delta-qQ^{q-1}+1$. Recall that the the limiting NLS optimizer $Q$ is non-degenerate~\cite{Kwong-89,Weinstein-85} and therefore $0$ is in the resolvent set of its restriction to the sector of radial functions. Using the non-degeneracy of $(\cL_Q)_{\rm rad}$ we can also add an exponential weight in the form
\begin{equation}
\norm{e^{c|x|}\left(\tilde u_\mu-Q+\mu^{\frac{p-q}{q-1}}(\cL_Q)_{\rm rad}^{-1}Q^p\right)}_{L^\ii(\R^d)}=o\left(\mu^{\frac{p-q}{q-1}}\right)
\label{eq:expansion_tilde_u_mu_exp}
\end{equation}
for all $c<\sqrt{\mu}$. Using~\eqref{eq:expansion_tilde_u_mu}  we find after scaling
\begin{align}
M(\mu)&=\mu^{\frac{4+d-dq}{2(q-1)}}\int_{\R^d}\widetilde u_\mu(x)^2\,dx \nn\\
&=\mu^{\frac{4+d-dq}{2(q-1)}}\int_{\R^d}Q(x)^2\,dx\nn\\
&\qquad -2\mu^{\frac{2(p-q)+4+d-dq}{2(q-1)}}\pscal{Q,(\cL_Q)^{-1}_{\rm rad}Q^p}+o\left(\mu^{\frac{2(p-q)+4+d-dq}{2(q-1)}}\right).\label{eq:expansion_M_sub-critical}
\end{align}
In the NLS case we can compute
$$\big(-\Delta-qQ^{q-1}+1\big)\left(\frac{rQ'}{2}+\frac{Q}{q-1}\right)=-Q$$
so that 
\begin{equation}
(\cL_Q)^{-1}_{\rm rad}Q=-\frac{rQ'}{2}-\frac{Q}{q-1}
\label{eq:derivative_mu_NLS}
\end{equation}
and
\begin{align}
\pscal{Q,(\cL_Q)^{-1}_{\rm rad}Q^p}&=\pscal{-\frac{rQ'}{2}-\frac{Q}{q-1},Q^p}\nn\\
&=-\frac{2(p+1)-d(q-1)}{2(p+1)(q-1)}\int_{\R^d}Q(x)^{p+1}\,dx. \label{eq:comput_Q_power}
\end{align}
Inserting in~\eqref{eq:expansion_M_sub-critical} we find~\eqref{eq:limit_M_mu_0}. 
The derivative equals
$$M'(\mu)=-2\mu^{\frac{4+d-dq}{2(q-1)}-1}\pscal{\tilde u_\mu,\tilde\cL_1(\mu)^{-1}_{\rm rad}\tilde u_\mu}$$
where
$$\tilde\cL_1(\mu)(\mu):=-\Delta+p\mu^{\frac{p-q}{q-1}}\tilde u_\mu^{p-1}-q\tilde u_\mu^{q-1}+1.$$
Due to the convergence of $\tilde u_\mu$ in $L^\ii(\R^d)$, the operator $\tilde\cL_1(\mu)$ converges to $\cL_Q:=-\Delta-qQ^{q-1}+1$ in the norm resolvent sense. Since $0$ is in the resolvent set, we obtain the convergence
$$\tilde\cL_1(\mu)^{-1}_{\rm rad}\to (\cL_Q)^{-1}_{\rm rad}=\big(-\Delta-qQ^{q-1}+1\big)^{-1}_{\rm rad}$$
in norm. More precisely, from the resolvent expansion we have
\begin{multline*}
\bigg\|\tilde\cL_1(\mu)^{-1}_{\rm rad}- (\cL_Q)^{-1}_{\rm rad}\\
+\mu^{\frac{p-q}{q-1}}(\cL_Q)^{-1}_{\rm rad}\Big(pQ^{p-1}+q(q-1)Q^{q-2}\delta\Big)(\cL_Q)^{-1}_{\rm rad}\bigg\|=O\left(\mu^{2\frac{p-q}{q-1}}\right) 
\end{multline*}
where $\delta=(\cL_Q)^{-1}_{\rm rad}Q^p$ and where the function in the parenthesis is understood as a multiplication operator. Recall from~\eqref{eq:expansion_tilde_u_mu_exp} that $\delta$ decreases exponentially at the same rate as $Q$, hence $Q^{q-2}\delta$ tends to 0 at infinity even when $q<2$. This shows that 
\begin{align}
\mu^{1-\frac{4+d-dq}{2(q-1)}}M'(\mu)&= -2\pscal{Q,(\cL_Q)^{-1}_{\rm rad}Q} +2\mu^{\frac{p-q}{q-1}}\bigg\{2\pscal{Q,(\cL_Q)_{\rm rad}^{-2}Q^p}\nn\\
&\qquad+p\pscal{(\cL_Q)_{\rm rad}^{-1}Q,Q^{p-1}(\cL_Q)_{\rm rad}^{-1}Q}\nn\\
&\qquad +q(q-1)\pscal{(\cL_Q)_{\rm rad}^{-1}Q,Q^{q-2}\delta(\cL_Q)_{\rm rad}^{-1}Q}\bigg\}+o\left(\mu^{\frac{p-q}{q-1}}\right).\label{eq:horrible_expansion_derivative_M}
\end{align}
By integration over $\mu$ and comparing with the expansion of $M(\mu)$, the two terms have to be given by the expression in~\eqref{eq:limit_M_derivative_mu_0}. For the first this is easy to check since by~\eqref{eq:derivative_mu_NLS}, we have
$$-2\pscal{Q,(\cL_Q)^{-1}_{\rm rad}Q}=2\pscal{Q,\frac{Q}{q-1}+\frac{rQ'}{2}}=\frac{4+d-dq}{2(q-1)}\int_{\R^d}Q(x)^2\,dx.$$
For the second term this is more cumbersome but the value can be verified as follows. Let us introduce the scaled function
$$Q_\mu(x)=\mu^{\frac1{q-1}}Q(\sqrt\mu x)$$
which solves the equation
$$-\Delta Q_\mu-Q_\mu^q+\mu Q_\mu=0.$$
Then we have $\partial_\mu Q_\mu=-(\cL_{Q_\mu})^{-1}Q_\mu$ where $\cL_{Q_\mu}=-\Delta-q Q_\mu^{q-1}+\mu$ and
$$\int_{\R^d}\frac{Q_\mu^{p+1}}{p+1}=\mu^{\frac{2(p+1)-d(q-1)}{2(q-1)}}\int_{\R^d}\frac{Q^{p+1}}{p+1},$$
$$\partial_\mu\int_{\R^d}\frac{Q_\mu^{p+1}}{p+1}=\pscal{\partial_\mu Q_\mu,Q_\mu^p}=-\pscal{(\cL_{Q_\mu})^{-1}Q_\mu,Q_\mu^p},$$
\begin{multline*}
\partial^2_{\mu}\int_{\R^d}\frac{Q_\mu^{p+1}}{p+1}=2\pscal{(\cL_{Q_\mu})^{-2}Q_\mu,Q_\mu^p}+p\pscal{(\cL_{Q_\mu})^{-1}Q_\mu,Q_\mu^{p-1}(\cL_{Q_\mu})^{-1}Q_\mu}\\
+q(q-1)\pscal{Q_\mu^{q-2}\big|(\cL_{Q_\mu})^{-1}Q_\mu\big|^2,(\cL_{Q_\mu})^{-1}Q_\mu^p}.
\end{multline*}
At $\mu=1$ we obtain
$$-\pscal{(\cL_{Q_\mu})^{-1}Q_\mu,Q_\mu^p}=\frac{2(p+1)-d(q-1)}{2(q-1)(p+1)}\int_{\R^d}Q^{p+1}$$
which is exactly~\eqref{eq:comput_Q_power} and 
\begin{multline*}
2\pscal{\cL_Q^{-2}Q,Q^p}+p\pscal{\cL_{Q}^{-1}Q,Q^{p-1}\cL_{Q}^{-1}Q}+q(q-1)\pscal{Q^{q-2}\big|\cL_{Q}^{-1}Q\big|^2,\delta}\\
=\frac{(2(p-1)+4+d-dq)(2(p-q)+4+d-dq)}{4(p+1)(q-1)^2}\int_{\R^d}Q^{p+1} 
 \end{multline*}
which gives exactly the equality between the second order terms in~\eqref{eq:horrible_expansion_derivative_M} and~\eqref{eq:limit_M_derivative_mu_0}.

\subsection{$M$ in the super-critical case} We have $u_\mu\to u_0$ strongly in the homogeneous Sobolev space $\dot{H}^1(\R^d)$ and in $C^2(\R^d)$ by~\cite{MorMur-14}. The limit $u_0$ is unique~\cite{KwoMcLPelTro-92} and it  is not in $L^2(\R^d)$ in dimensions $d=3,4$, therefore $M(\mu)$ cannot have a bounded subsequence and the limit~\eqref{eq:limit_M_mu_0_supercritical} follows in this case. Let us prove the strong convergence in $L^2(\R^d)$ when $d\geq5$. We use~\cite[Lem. A.III]{BerLio-83} which says that 
$$u(x)\leq C_d\frac{\norm{\nabla u}_{L^2(\R^d)}}{|x|^{\frac{d-2}{2}}},\qquad |x|\geq 1$$
for a universal constant $C_d$. Due to the strong convergence of $u_\mu$ in $\dot{H}^1(\R^d)$, the gradient term is bounded and this gives
$$\left(-\Delta-\frac{C}{|x|^{\frac{(d-2)(q-1)}2}}\right)u_\mu\leq (-\Delta+\mu+u_\mu^{p-1}-u_\mu^{q-1})u_\mu=0,\qquad|x|\geq1$$
for a constant $C$. We can also use that 
$$\left(-\Delta-\frac{C}{|x|^{\frac{(d-2)(q-1)}2}}\right)\frac1{|x|^{d-2-\eps}}=\left(\eps(d-2-\eps)-\frac{C}{|x|^{\frac{(d-2)(q-1)-4}2}}\right)\frac1{|x|^{d-\eps}}$$
which is positive for $|x|$ large enough since $q-1>4/(d-2)$. By the maximum principle on $\R^d\setminus B_R$, we deduce that 
\begin{equation}
u_\mu(x)\leq \frac{u_\mu(R)R^{d-2-\eps}}{|x|^{d-2-\eps}},\qquad \forall |x|\geq R.
\label{eq:estim_max_principle}
\end{equation}
When $\eps$ is small enough, the domination is in $L^2(\R^d)$ for $d\geq5$ and this shows, by the dominated convergence theorem, that $u_\mu\to u_0$ strongly in $L^2(\R )$. The behavior of $M'$ is discussed below in Section~\ref{sec:proof_super_critical}. 

\subsection{$M$ in the critical case}  This case is more complicated and was studied at length in~\cite{MorMur-14}. The function $w_\mu$ in~\eqref{eq:rescaled_critical} satisfies the equation
\begin{equation}
 -\Delta w_\mu +(\eps_\mu)^{\frac{(p-1)(d-2)-4}{2}}w_\mu^{p}-w_\mu^q+\frac\mu{\eps_\mu^2} w_\mu=0
 \label{eq:effective_critical}
\end{equation}
for $\sqrt\mu\ll \eps_\mu\ll1$ as in~\eqref{eq:def_eps_mu} and we have 
$$M(\mu)=\frac{1}{\eps_\mu^{2}}\int_{\R^d}w_\mu(x)^2\,dx.$$
In dimensions $d=3,4$, we have $\|w_\mu\|_{L^2}\to\ii$ because the limit $S$ is not in $L^2$. In dimensions $d\geq5$, it is proved in~\cite[Lem.~4.8, Cor.~1.14]{MorMur-14} that $\|w_\mu\|_{L^2}\sim1$ and this implies $M(\mu)\to\ii$ in all cases. The same argument as in~\eqref{eq:estim_max_principle} actually gives $w_\mu\to S$ in $L^2(\R^d)$ for $d\geq5$. 

\subsection{An upper bound on $M'$} Next we derive an upper bound on $M'(\mu)$ following ideas from~\cite{KilOhPocVis-17}. 

\begin{lemma}[An estimate on $M'(\mu)$]\label{lem:M_derivative}
Let $d\geq2$ and $p>q>1$. Then we have 
\begin{multline}
\frac{M'(\mu)}{2}\left(\frac{(p-1)(p-q)(d-2)}{p+1}\beta(\mu)+\frac{2}{d}\big(d+2-(d-2)q\big)\right)\\
<\frac{dM(\mu)^2}{2T(\mu)}\left(\frac{d(p-1)(p-q)}{2(p+1)}\beta(\mu)+1+\frac4d -q\right)
 \label{eq:estim_M_derivative}
\end{multline}
for 
$$\begin{cases}
\text{all $\mu>0$ if $d=2$ or if $d\geq3$ and $q\leq 1+\frac{4}{d-2}$,}\\
\text{small enough $\mu>0$ if $d\geq3$ and $q>1+\frac{4}{d-2}$,}\\
\end{cases}$$
where
$$T(\mu)=\int_{\R^d}|\nabla u_\mu(x)|^2\,dx,\qquad \beta(\mu)=T(\mu)^{-1}\int_{\R^d}u_\mu(x)^{p+1}\,dx.$$
\end{lemma}

\begin{proof}[Proof of Lemma~\ref{lem:M_derivative}]
The Pohozaev identity~\eqref{pohozaev} gives
\begin{align}
 \frac{d-2}{2d}T(\mu)&=-\frac1{p+1}\int_{\R^d} u_\mu(x)^{p+1}\,dx+\frac1{q+1}\int_{\R^d} u_\mu(x)^{q+1}\,dx\nn\\
 &\qquad -\frac\mu2\int_{\R^d}u_\mu(x)^2\,dx\nn\\
 &=-\frac{T(\mu)\beta(\mu)}{p+1}+\frac1{q+1}\int_{\R^d} u_\mu(x)^{q+1}\,dx-\frac{\mu M(\mu)}2
 \label{eq:Pohozaev_2PNLS}
\end{align}
and taking the scalar product with $u$ in~\eqref{powernl} we find
\begin{align}
 T(\mu)&=-\int_{\R^d} u_\mu(x)^{p+1}\,dx+\int_{\R^d} u_\mu(x)^{q+1}\,dx -\mu\int_{\R^d}u_\mu(x)^2\,dx\nn\\
 &=-T(\mu)\beta(\mu)+\int_{\R^d} u_\mu(x)^{q+1}\,dx -\mu M(\mu).
\end{align}
This gives the relation
\begin{equation}
 \frac{p-q}{(p+1)(q+1)}\beta(\mu)=\frac{d-2}{2d}-\frac{1}{q+1}+\frac{(q-1)}{2(q+1)}\frac{\mu\,M(\mu)}{T(\mu)}.
 \label{eq:formule_beta}
\end{equation}
When $d\geq3$ and $q>1+4/(d-2)$ we have 
\begin{equation}
 \lim_{\mu\searrow0}\beta(\mu)=\frac{(p+1)(d-2)}{2d(p-q)}\left(q-1-\frac{4}{d-2}\right):=\beta(0).
 \label{eq:beta_0}
\end{equation}
In all the other cases we have $\beta(\mu)\to0$ when $\mu\searrow0$. More precisely, we have 
$$\beta(\mu)=O\left(\mu^{\frac{p-q}{q-1}}\right)\to0\qquad\text{for $d=2$, or $d\ge3$ and $q<1+\frac{4}{d-2}$}$$
and
$$\beta(\mu)=O\left(\eps_\mu^{\frac{d-2}{2}(p-q)}\right)\to0\qquad\text{for $d\ge3$ and $q=1+\frac4{d-2}$}.$$
Next we compute the symmetric matrix $L_{ij}$ of the restriction of the operator $\cL_\mu$ to the finite dimensional space spanned by $u_\mu$, $\partial_\mu u_\mu$ and 
$$\frac{x\cdot\nabla+\nabla\cdot x}{2}u_\mu=\left(x\cdot\nabla+\frac{d}2\right)u_\mu=ru_\mu'+\frac{d}{2}u_\mu.$$
Some tedious but simple computations using the above relations give
$$L_{11}:=\pscal{\partial_\mu u_\mu,\cL_\mu \partial_\mu u_\mu}=-\frac{M'(\mu)}{2},\qquad L_{12}:=\pscal{\partial_\mu u_\mu,\cL_\mu u_\mu}=-M(\mu),$$
$$L_{22}:=\pscal{u_\mu,\cL_\mu u_\mu}=\left(\frac{(p-1)(p-q)}{p+1}\beta(\mu)-\frac{2(q+1)}{d}\right)T(\mu),$$
$$L_{13}:=\pscal{ru_\mu'+\frac{d}{2}u_\mu,\cL_\mu \partial_\mu u_\mu}=0,$$
$$L_{23}:=\pscal{ru_\mu'+\frac{d}{2}u_\mu,\cL_\mu u_\mu}=\left(\frac{d(p-1)(p-q)}{2(p+1)}\beta(\mu)-(q-1)\right)T(\mu)$$
and 
\begin{align*}
L_{33}&:=\pscal{ru_\mu'+\frac{d}{2}u_\mu,\cL_\mu \left(ru_\mu'+\frac{d}{2}u_\mu\right)}\\
&=\frac{d}2\left(\frac{d(p-1)(p-q)}{2(p+1)}\beta(\mu)+1+\frac4d -q\right)T(\mu). 
\end{align*}
Note that $L_{22}<0$ for $\mu$ small enough when $d=2$ or when $d\ge 3$ and $q\leq 1+4/(d-2)$. We have
\begin{multline*}
\det\begin{pmatrix}
L_{22}&L_{23}\\
L_{32}&L_{33}\end{pmatrix}\\
=T(\mu)^2\left(-\frac{(p-1)(p-q)(d-2)}{p+1}\beta(\mu)+\frac{2}d \left(q(d-2)-d-2\right)\right).
\end{multline*}
This is always negative in the sub-critical and critical case. In the super-critical case, we obtain from~\eqref{eq:beta_0} that 
\begin{equation}
\det\begin{pmatrix}
L_{22}&L_{23}\\
L_{32}&L_{33}\end{pmatrix}\\
\underset{\mu\searrow0}\sim-T(\mu)^2\frac{\big(q(d-2)-d-2\big)\big(p(d-2)-d-2\big)}{2d}<0.
\label{eq:determinant_small}
\end{equation}
Since $\cL_\mu$ has a unique negative eigenvalue, we conclude that the full determinant is negative:
\begin{align*}
0&>\det(L)\\
&=\frac{M'(\mu)}{2}\left(\frac{(p-1)(p-q)(d-2)}{p+1}\beta(\mu)+\frac{2}{d}\big(d+2-(d-2)q\big)\right)T(\mu)^2\\
&\qquad -\frac{d}2M(\mu)^2\left(\frac{d(p-1)(p-q)}{2(p+1)}\beta(\mu)+1+\frac4d -q\right)T(\mu).
\end{align*}
This gives the estimate~\eqref{eq:estim_M_derivative}.
\end{proof}

In the critical case $q=1+4/(d-2)$,~\eqref{eq:estim_M_derivative} gives for $\mu$ small enough
$$M'(\mu)<-c\frac{M(\mu)^2}{T(\mu)\beta(\mu)}\underset{\mu\searrow0}{\sim}
-c\begin{cases}
\eps_\mu^{-\frac{(d-2)p-d-2}{2}}M(\mu)^2&\text{for $d=3,4$,}\\
\eps_\mu^{-\frac{(d-2)(p-1)+4}{2}}&\text{for $d\geq5$.}
\end{cases}
$$
We obtain $M'(\mu)\to-\ii$ as was claimed in the statement. 

In the super-critical case $q>1+4/(d-2)$, we similarly find $M'(\mu)\to-\ii$ in dimensions $d=3,4$. In dimensions $d\geq5$, using~\eqref{eq:beta_0} we find $M'(\mu)<0$ for $\mu$ small enough whenever
$$p<1+\frac{4}{d-2}+\frac{32}{d(d-2)\big(q(d-2)-(d+2)\big)}.$$

\subsection{$M'$ in the super-critical case}\label{sec:proof_super_critical}
This last section is devoted to the study of $M'$ in the super-critical case. We have seen that $u_\mu\to u_0$ in $\dot{H}^1(\R^d)\cap L^\ii(\R^d)$, the unique radial-decreasing solution to the equation
$$-\Delta u_0+u_0^p-u_0^q=0.$$
In addition, the convergence holds in $L^s(\R^d)$ for all $s>d/(d-2)$ since $u_\mu(r)\leq Cr^{2-d-\eps}$ at infinity. This includes $L^2(\R^d)$ only in dimensions $d\geq5$. 

In Lemma~\ref{lem:non-degenerate_mu_0} in Appendix~\ref{app:u_0}, we show that the limiting linearized operator
$$\cL_0:=-\Delta +p(u_0)^{p-1}-q(u_0)^{q-1}$$
has the trivial kernel 
$$\ker(\cL_0)=\left\{\partial_{x_1}u_0,...,\partial_{x_d}u_0\right\}.$$
This allows us to define $(\cL_0)_{\rm rad}^{-1}$ by the functional calculus. Note that $\cL_0$ admits exactly one negative eigenvalue, since
$$\pscal{u'_0,\cL_0u_0'}=-(d-1)\int_{\R^d}\frac{u_0'(x)^2}{|x|^2}\,dx<0$$
and since it is the norm-resolvent limit of $\cL_\mu$, which has exactly one negative eigenvalue. In particular, we see that $(\cL_0)_{\rm rad}^{-1}$ has one negative eigenvalue and is otherwise positive and unbounded from above. Our first step is to prove that its quadratic form domain is the same as for the free Laplacian, in sufficiently high dimensions. In the whole section we assume $d\geq5$ for simplicity, although some parts of our proof apply to $d\in\{3,4\}$. 

\begin{lemma}[Quadratic form domain of $(\cL_0)_{\rm rad}^{-1}$]\label{lem:quadratic_form_domain_inverse_L}
Let $d\geq5$ and $p>q>1+4/(d-2)$. 
There exists a constant $C>0$ such that 
$$\frac1C(-\Delta)_{\rm rad}^{-1}-C\leq (\cL_0)_{\rm rad}^{-1}\leq C(-\Delta)_{\rm rad}^{-1}$$
in the sense of quadratic forms.
\end{lemma}

\begin{proof}
In the proof we remove the index `rad' and use the convention that all the operators are restricted to the sector of radial functions. Note that the following arguments work the same on $L^2(\R^d)$ for a general potential $V$ such that $-\Delta+V$ has no zero eigenvalue. But this of course not the case of our linearized operator $\cL_0$ which always has $\partial_{x_j}u_0$ in its kernel. Introducing the notation
$V_0:=pu_0^{p-1}-qu_0^{q-1}$
for the external potential, we start with the relation
\begin{equation}
 \cL_0=(-\Delta)^{\frac12}\left(1+(-\Delta)^{-\frac12}V_0(-\Delta)^{-\frac12}\right)(-\Delta)^{\frac12}.
\label{eq:Birman-Schwinger-type}
\end{equation}
Recalling that $u_0(r)\sim C_0r^{2-d}$ at infinity and that $q>1+4/(d-2)$, we obtain 
\begin{equation}
V_0\in L^s(\R^d),\qquad \forall s>\frac{d}{(d-2)(q-1)}>\frac{d}{4}.
\label{eq:prop_V_super_critical}
\end{equation}
In particular we have  $V_0\in L^{d/2}(\R^d)$. From the Hardy-Littlewood-Sobolev (HLS) inequality~\cite{LieLos-01}, we then know that the operator 
\begin{equation}
K:=(-\Delta)^{-\frac12}V_0(-\Delta)^{-\frac12}
\label{eq:def_K}
\end{equation}
is self-adjoint and compact on $L^2(\R^d)$, with 
$\norm{K}\leq C\norm{V_0}_{L^{d/2}(\R^d)}.$
The fact that $\ker(\cL_0)_{\rm rad}=\{0\}$ is equivalent to the property that 
$-1\notin \sigma\left(K\right)$
where we recall that $K$ is here only considered within the sector of radial functions. Let $g$ be a radial eigenfunction of $K$, corresponding to a discrete eigenvalue $\lambda\neq0$,
$$Kg=(-\Delta)^{-\frac12}V_0(-\Delta)^{-\frac12}g=\lambda g.$$
Multiplying by $(-\Delta)^{-1}$ on the left we find
\begin{equation}
 (-\Delta)^{-1}g=\lambda^{-1} (-\Delta)^{-\frac32}V_0(-\Delta)^{-\frac12}g.
 \label{eq:estim_eigenfn_K}
\end{equation}
Using this time $d>4$ and again the HLS inequality, we see that the operator $(-\Delta)^{-3/2}V_0(-\Delta)^{-1/2}$ is compact with 
$$\norm{(-\Delta)^{-\frac32}V_0(-\Delta)^{-\frac12}}\leq C\norm{V_0}_{L^{d/4}(\R^d)}.$$
This proves that $(-\Delta)^{-1}g\in L^2(\R^d)$. Next we go back to~\eqref{eq:Birman-Schwinger-type}. Since $\cL_0$ has one negative eigenvalue, this implies that $K$ has at least one eigenvalue $<-1$, for otherwise $\cL_0$ would be positive by~\eqref{eq:Birman-Schwinger-type}. On the other hand, $K$ cannot have two eigenvalues $<-1$ otherwise after testing against this subspace and using that the corresponding eigenfunctions are in the domain of $(-\Delta)^{-1/2}$, we would find that $\cL_0$ is negative on a subspace of dimension 2. We conclude that $K$ has exactly one simple eigenvalue $\nu_1<-1$ within the radial sector and call the corresponding normalized eigenfunction $g_1$. Next we invert~\eqref{eq:Birman-Schwinger-type} and obtain 
\begin{equation}
  (\cL_0)_{\rm rad}^{-1}=(-\Delta)^{-\frac12}\left(1+K\right)^{-1}(-\Delta)^{-\frac12}.
 \label{eq:inverting_L_0}
\end{equation}
We have 
$$\frac{1}{1+\nu_2}-\left(\frac{1}{1+\nu_2}-\frac{1}{1+\nu_1}\right)|g_1\rangle\langle g_1|\leq \frac1{1+K}\leq \frac{1}{1+\nu_2}$$
with $\nu_2=\min\sigma(K)\setminus\{\nu_1\}>-1$. This gives 
\begin{equation}
\frac{(-\Delta)^{-1}}{1+\nu_2}-c\norm{(-\Delta)^{-\frac12}g_1}^{2}\leq  (\cL_0)_{\rm rad}^{-1}\leq \frac{(-\Delta)^{-1}}{1+\nu_2}
\label{eq:compare_L_Delta_proof}
\end{equation}
with
$$c=\frac{1}{1+\nu_2}-\frac{1}{1+\nu_1}>0.$$
The norm on the left of~\eqref{eq:compare_L_Delta_proof} is finite since we have even proved that $(-\Delta)^{-1}g_1\in L^2(\R^d)$ and this concludes the proof of the lemma. One can actually show that the domains of $(\cL_0)^{-1}_{\rm rad}$ and $(-\Delta)^{-1}_{\rm rad}$ coincide, not just the quadratic form domains, but this is not needed in our argument. 
\end{proof}

Next we turn to the upper  bound on $M'(\mu)$. 

\begin{lemma}[Upper bound on the limit of $M'(\mu)$]
Let $d\geq5$ and $p>q>1+{4}/{(d-2)}$. Then we have 
\begin{equation}
 \limsup_{\mu\searrow0}M'(\mu)\leq -2\pscal{u_0,(\cL_0)^{-1}_{\rm rad}u_0}\in[-\ii,\ii),
 \label{eq:limsup_derivative_super_critical}
\end{equation}
interpreted in the sense of quadratic forms. In dimensions $d=5,6$ the right side equals $-\ii$ whereas it is finite in dimensions $d\geq7$. 
\end{lemma}

From the proof of Lemma~\ref{lem:quadratic_form_domain_inverse_L}, the proper interpretation of the quadratic form on the right side of~\eqref{eq:limsup_derivative_super_critical} is
$$\pscal{u_0,(\cL_0)^{-1}_{\rm rad}u_0}:=\pscal{(-\Delta)^{-\frac12}u_0,(1+K)^{-1}(-\Delta)^{-\frac12}u_0}$$
where $K$ is given by~\eqref{eq:def_K}.

\begin{proof}
We start with
\begin{align*}
 M'(\mu)= -2\pscal{u_\mu,(\cL_\mu)^{-1}_{\rm rad}u_\mu}&=-\frac{2|\pscal{f_{1,\mu},u_\mu}|^2}{\lambda_1(\mu)} -2\pscal{u_\mu,(\cL_\mu)_+^{-1}u_\mu}\\
&\leq-\frac{2|\pscal{f_{1,\mu},u_\mu}|^2}{\lambda_1(\mu)} -2\pscal{u_\mu,(\cL_\mu+\eps)_+^{-1}u_\mu}.
\end{align*}
Here $f_{1,\mu}$ is the unique eigenfunction corresponding to the negative eigenvalue $\lambda_1(\mu)$ of $\cL_\mu$ and $\eps$ is a small fixed number, chosen so that $\eps\leq |\lambda_1(\mu)|/2$. From the convergence of $\cL_\mu$ in the norm-resolvent sense and the strong convergence of $u_\mu$ in $L^2(\R^d)$, we obtain in the limit 
$$\limsup_{\mu\searrow0} M'(\mu)\leq -\frac{2|\pscal{f_{1},u_0}|^2}{\lambda_1} -2\pscal{u_0,(\cL_0+\eps)_+^{-1}u_0}$$
with of course $\cL_0f_1=\lambda_1 f_1$ and $\lambda_1<0$. Letting finally $\eps\to0^+$, this gives the limit~\eqref{eq:limsup_derivative_super_critical}.

From Lemma~\ref{lem:quadratic_form_domain_inverse_L} we know that the right side of~\eqref{eq:limsup_derivative_super_critical} is finite if and only if $\|(-\Delta)^{-1/2}u_0\|_{L^2}$ is finite. This turns out to be infinite in dimensions $d=5,6$ and finite in larger dimensions. The reason is the following. Since $r^{2-d}u_0(r)\to C_0>0$ at infinity, we can write $u_0=u_0\1_{B_1}+u_0\1_{\R^d\setminus B_1}:=u_1+u_2$ where $u_1\in (L^1\cap L^\ii)(\R^d)$ and 
$$\frac{c\1(|x|\geq1)}{|x|^{d-2}}\leq u_2(x)\leq \frac{C\1(|x|\geq1)}{|x|^{d-2}}.$$
But 
$$\norm{(-\Delta)^{-\frac12}u_0}^2_{L^2(\R^d)}=c \iint_{\R^d\times\R^d}\frac{u_0(x)u_0(y)}{|x-y|^{d-2}}\,dx\,dy$$
and the terms involving $u_1$ are finite by the HLS inequality, whereas the term involving $u_2$ twice is comparable to 
$$\iint_{\substack{|x|\geq1\\ |y|\geq1}}\frac{dx\,dy}{|x|^{d-2}|y|^{d-2}|x-y|^{d-2}}=\begin{cases}
+\ii&\text{when $d=5,6$,}\\
<\ii&\text{when $d\geq7$.}
\end{cases}
$$
This concludes the proof of the lemma. 
\end{proof}

We are left with showing the limit in dimensions $d\geq7$. 

\begin{lemma}[Limit of $M'(\mu)$ in dimensions $d\geq7$]
Let $d\geq7$ and $p>q>1+{4}/{(d-2)}$. Then we have 
\begin{equation}
 \lim_{\mu\searrow0}M'(\mu)= -2\pscal{u_0,(\cL_0)^{-1}_{\rm rad}u_0}.
 \label{eq:lim_derivative_super_critical}
\end{equation}
\end{lemma}

\begin{proof}
Similarly to~\eqref{eq:Birman-Schwinger-type} we can write
\begin{equation}
 \cL_\mu=(-\Delta+\mu)^{\frac12}\left(1+K_\mu\right)(-\Delta+\mu)^{\frac12}
 \label{eq:Birman-Schwinger-type_mu}
\end{equation}
where
\begin{align*}
K_\mu&:=(-\Delta+\mu)^{-\frac12}\left(pu_\mu^{p-1}-qu_\mu^{q-1}\right)(-\Delta+\mu)^{-\frac12}\\
&=\left(\frac{-\Delta}{-\Delta+\mu}\right)^{\frac12}(-\Delta)^{-\frac12}\left(pu_\mu^{p-1}-qu_\mu^{q-1}\right)(-\Delta)^{-\frac12}\left(\frac{-\Delta}{-\Delta+\mu}\right)^{\frac12}.
\end{align*}
Since $pu_\mu^{p-1}-qu_\mu^{q-1}\to pu_0^{p-1}-qu_0^{q-1}=V_0$ in $L^{d/2}(\R^d)$, we have 
$$(-\Delta)^{-\frac12}\left(pu_\mu^{p-1}-qu_\mu^{q-1}\right)(-\Delta)^{-\frac12}\to K=(-\Delta)^{-\frac12}V_0(-\Delta)^{-\frac12}$$
in operator norm, by the HLS inequality. On the other hand, the operator $(-\Delta)^{1/2}(-\Delta+\mu)^{-1/2}$ is bounded by 1 and converges to the identity strongly. Since $K$ is compact, this shows that $K_\mu\to K$ in operator norm. In particular, the spectrum of $K_\mu$ converges to that of $K$ and, since $1+K$ is invertible, we conclude that $(1+K_\mu)^{-1}$ is bounded and converges in norm towards $(1+K)^{-1}$. This allows us to invert~\eqref{eq:Birman-Schwinger-type_mu} and obtain
$$(\cL_\mu)_{\rm rad}^{-1}=(-\Delta+\mu)^{-\frac12}\left(1+K_\mu\right)^{-1}(-\Delta+\mu)^{-\frac12}$$
as well as
$$M'(\mu)=-2\pscal{(-\Delta+\mu)^{-\frac12}u_\mu,\big(1+K_\mu\big)^{-1}(-\Delta+\mu)^{-\frac12}u_\mu}.$$
From the HLS inequality we have 
\begin{align*}
&\norm{(-\Delta+\mu)^{-\frac12}u_\mu- (-\Delta)^{-\frac12}u_0}_{L^2(\R^d)}\\
&\leq  \norm{(-\Delta+\mu)^{-\frac12}(u_\mu-u_0)}_{L^2(\R^d)}+\norm{\big((-\Delta+\mu)^{-\frac12}- (-\Delta)^{-\frac12}\big)u_0}_{L^2(\R^d)}\\
&\leq  C\norm{u_\mu-u_0}_{L^{\frac{2d}{d+2}}(\R^d)}+\norm{\big((-\Delta+\mu)^{-\frac12}- (-\Delta)^{-\frac12}\big)u_0}_{L^2(\R^d)}
\end{align*}
which tends to zero since $u_\mu$ converges in $L^{\frac{2d}{d+2}}(\R^d)$ and $(-\Delta)^{-1/2}u_0\in L^2(\R^d)$ for $d\geq7$. Thus we can pass to the limit and obtain
\begin{equation*}
\lim_{\mu\searrow0}M'(\mu)=-2\pscal{(-\Delta)^{-\frac12}u_0,\big(1+K\big)^{-1}(-\Delta)^{-\frac12}u_0},
\end{equation*}
where the right side is the definition of the quadratic form of $(\cL_0)^{-1}_{\rm rad}$.
\end{proof}

\begin{remark}[Higher derivatives]\label{rmk:higher_derivatives}
A calculation shows that 
\begin{equation}
 M''(\mu)=6\int_{\R^d}\delta_\mu^2-2p(p-1)\int_{\R^d}\delta_\mu^3u_\mu^{p-2}+2q(q-1)\int_{\R^d}\delta_\mu^3u_\mu^{q-2}
 \label{eq:derivee_seconde}
\end{equation}
with $\delta_\mu:=(\cL_\mu)^{-1}_{\rm rad}u_\mu$. From the ODE we have $\delta_\mu(r)\leq Cr^{4-d+\eps}$ for $r\geq1$ and this can be used to show that the two terms involving $\delta_\mu^3$ converge in dimensions $d\geq7$. In dimensions $d\in\{7,8\}$ the first term has to diverge because $\delta_0\notin L^2(\R^d)$. Thus we have 
$$\lim_{\mu\searrow0}M''(\mu)=\begin{cases}
+\ii&\text{if $d\in\{7,8\}$,}\\
M''(0)\in\R&\text{if $d\geq9$.}
\end{cases}$$
By induction, it is possible to prove that $(-1)^kM^{(k)}(\mu)\to+\ii$ for a sufficiently large $k$ in any dimension $d\geq3$. 
\end{remark}

This concludes the proof of Theorem~\ref{thm:limit_mu_0}.\qed

\section{Proof of Theorem~\ref{thm:limit_mu_star} on the limit $\mu\to\mu_*$}\label{sec:proof_mu_star}

Throughout the proof we set $G_*:=G_{\mu_*}$ and $g_*:=g_{\mu_*}$. 

\subsection{Local convergence}
Our first step is to prove that $u_\mu$ almost satisfies the one-dimensional equation of $U_*$ and to prove the local convergence $u_\mu(r)\to\beta_*$ for any fixed $r$, which was claimed after Theorem~\ref{thmpowernl}.

\begin{lemma}[Local convergence]
We have 
\begin{equation}
\norm{(u_\mu')^2+2G_*(u_\mu)}_{L^\ii(\R_+)}\leq \mu_*-\mu,
 \label{eq:limit_eq}
\end{equation}
\begin{equation}
\norm{u_\mu'+\sqrt{-2G_*(u_\mu)}}_{L^\ii(\R_+)}\leq \sqrt{\mu_*-\mu},
 \label{eq:estim_square_root}
\end{equation}
and
\begin{equation}
\lim_{\mu\to\mu_*}u_\mu(r)=\beta_*,\qquad \lim_{\mu\to\mu_*}u'_\mu(r)=0. 
\end{equation}
uniformly on any compact interval $[0,R]$. 
\end{lemma}

We recall that $G_*$ is negative on $\R_+$ by definition of $\mu_*$. 

\begin{proof}
We start with the ODE
\begin{equation}
u_\mu''+(d-1)\frac{u_\mu'}{r}+g_\mu(u_\mu)=0.
 \label{eq:ODE}
\end{equation}
Multiplying by $u'_\mu$, we find
$$\frac{(u_\mu')^2}{2}+(d-1)\int_0^r\frac{u_\mu'(s)^2}{s}\,ds+G_\mu(u_\mu)=G_\mu(u_\mu(0)).$$
Evaluating at $r=+\ii$, this gives
\begin{equation}
 (d-1)\int_0^\ii\frac{u_\mu'(s)^2}{s}\,ds=G_\mu(u_\mu(0))=G_*(u_\mu(0))+\frac{\mu_*-\mu}{2}u_\mu(0)^2\leq\frac{\mu_*-\mu}{2}
 \label{eq:limit_d_term}
\end{equation}
since $G_*:=G_{\mu_*}\leq0$ and $0\leq u_\mu\leq1$. 
We can also rewrite the equation in the form
\begin{equation}
\frac{(u_\mu')^2}{2}+G_*(u_\mu)=(d-1)\int_r^\ii\frac{u_\mu'(s)^2}{s}\,ds-\frac12(\mu_*-\mu)u_\mu^2.
\label{eq:integrated_ODE}
\end{equation}
Due to~\eqref{eq:limit_d_term} we obtain~\eqref{eq:limit_eq}. Noticing that $|a^2-b^2|\leq \eps^2$ implies $|a-b|\leq \eps$ whenever $a,b\geq0$, we obtain~\eqref{eq:estim_square_root}. We also have 
\begin{equation}
|u_\mu(r)-u_\mu(0)|\leq \int_0^r|u'_\mu(s)|\,ds\leq \frac{r}{\sqrt2}\left(\int_0^r\frac{u'_\mu(s)^2}{s}\,ds\right)^{\frac12}\leq \frac{r\sqrt{\mu_*-\mu}}{2\sqrt{d-1}}
\label{eq:estim_uniform_local}
\end{equation}
and therefore we obtain the local convergence of $u_\mu$ to $\beta_*$, uniformly on any compact interval $[0,R]$. For the derivative we use~\eqref{eq:estim_square_root}.
\end{proof}

\subsection{Convergence of $u_\mu(\cdot+R_\mu)$}
Next we look at $u_\mu$ much further away. We fix $\gamma\in(0,\beta_*)$ and define $R_\mu$ like in the statement by the condition that 
$u_\mu(R_\mu)=\gamma$, for $\mu$ close enough to $\mu_*$. We then introduce 
$$v_\mu(r):=u_\mu(R_\mu+r),\qquad r\in[-R_\mu,\ii).$$
From~\eqref{eq:estim_uniform_local} we know that 
\begin{equation}
 R_\mu\geq \frac{2\sqrt{d-1}(u_\mu(0)-\gamma)}{\sqrt{\mu_*-\mu}}\to\ii.
 \label{eq:lower_bd_R_mu}
\end{equation}
The function $v_\mu$ satisfies a relation similar to~\eqref{eq:integrated_ODE}. It is uniformly bounded together with its derivative and we can pass to the uniform local limit $\mu\to\mu_*$, possibly after extraction of a subsequence. We obtain in the limit that $U_*=\lim_{\mu\to\mu_*}v_\mu$ solves~\eqref{eq:U}. That is, $U_*$ is the unique unstable solution of the $d=1$ Hamiltonian system, linking the two stationary points $\beta_*$ and $0$ and passing through $\gamma$ at $r=0$. More precisely, we have since $U_*'<0$
$$-\frac{U_*'}{\sqrt{2|G_*(U_*)|}}=1.$$
Therefore
$$U_*(r)=\Psi^{-1}(r),\qquad \Psi(v)=-\int_{\gamma}^{v}\frac{ds}{\sqrt{2|G_*(s)|}}.$$ 
Note that $\Psi$ diverges logarithmically at 0 and $\beta_*$, so that $U_*$ converges exponentially fast towards $\beta_*$ at $-\ii$ and $0$ at $+\ii$. To summarize the situation, we have for every $h>0$
\begin{multline}
\lim_{\mu\to\mu_*}\norm{u_\mu-U_*(\cdot-R_\mu)}_{L^\ii(R_\mu-h,R_\mu+h)}\\
=\lim_{\mu\to\mu_*}\norm{u'_\mu-U_*'(\cdot-R_\mu)}_{L^\ii(R_\mu-h,R_\mu+h)}=0
\label{eq:CV_Lii_local}
\end{multline}
where the convergence of the derivatives follows from~\eqref{eq:estim_square_root}. 

In the next lemma we derive pointwise bounds on $v_\mu=u_\mu(\cdot+R_\mu)$ and its derivatives which will later allow us to improve the limit~\eqref{eq:CV_Lii_local}. 

\begin{lemma}[Pointwise exponential bounds]\label{lem:exp_bounds}
We have the bounds
\begin{equation}
 v_\mu(r)\begin{cases}
\dps\leq Ce^{-c|r|}&\text{on $[0,\ii)$,}\\[0.3cm]
\dps\geq \beta_\mu-Ce^{-c|r|}&\text{on $[-R_\mu,0]$,}\\
  \end{cases}
 \label{eq:exp_decay}
\end{equation}
and
\begin{equation}
 |v'_\mu(r)|+|v''_\mu(r)|\leq Ce^{-c|r|}\qquad\text{on $[-R_\mu,\ii)$.}
 \label{eq:exp_decay_derivative}
\end{equation}
for some $c,C>0$ independent of $\mu\in[\mu_*/2,\mu_*)$. 
\end{lemma}

\begin{proof}
Due to the local uniform convergence of $u_\mu$ around $R_\mu$ and the fact that $u_\mu$ is decreasing, we deduce that for all $\eps>0$ we can find an $h>0$ such that 
$$u_\mu\begin{cases}
\geq\beta_*-\eps &\text{on $[0,R_\mu-h]$}\\
\leq\eps &\text{on $[R_\mu+h,\ii)$}        
       \end{cases}.$$
We have $g_*'(0)=-\mu_*<0$ and $g_*'(\beta_*)<0$. Therefore, choosing $\eps$ small enough, we obtain 
$$g_\mu(u_\mu)\begin{cases}
\geq \frac{g_*'(\beta_*)}2 (u_\mu-\beta_\mu)&\text{on $[0,R_\mu-h]$,}\\
\leq-\frac{\mu_*}2 u_\mu&\text{on $[R_\mu+h,\ii)$.}
       \end{cases}$$
In other words, $u_\mu$ satisfies 
$$\begin{cases}
-\Delta (\beta_\mu-u_\mu)+c^2(\beta_\mu-u_\mu)\leq 0 &\text{on the ball $B_{R_\mu-h}$}\\
-\Delta u_\mu+c^2u_\mu\leq 0 &\text{on $\R^d\setminus B_{R_\mu+h}$.}
       \end{cases}$$
for $c^2=\min(\mu_*,|g_*'(\beta_*)|)/2$. Next we recall that 
$$(-\Delta+c^2)e^{-\alpha |x|}=\left(c^2-\alpha^2+\frac{d-1}{|x|}\alpha\right)e^{-\alpha|x|}$$
in the sense of distributions on $\R^d$. On $\R^d\setminus B_{R_\mu+h}$ we choose $\alpha=c$ and obtain that 
$$\left(-\Delta +c^2\right)\left(u_\mu-u_\mu(R_\mu+h)e^{-c(|x|-R_\mu-h)}\right)\leq 0 \qquad\text{on $\R^d\setminus B_{R_\mu+h}$.}$$
By the maximum principle we obtain
$$u_\mu\leq u_\mu(R_\mu+h)e^{-c(|x|-R_\mu-h)}\leq Ce^{-c(|x|-R_\mu)}$$
with $C=e^{ch}$. This is the first bound in~\eqref{eq:exp_decay} on $\R^d\setminus B_{R_\mu+h}$. 
To prove the estimate on $B_{R_\mu-h}$, we recall that 
$$(-\Delta+c^2)\frac{e^{\alpha |x|}}{|x|^{\frac{d-1}2}}=\left(c^2-\alpha^2+\frac{(d-1)(d-3)}{4|x|^2}\right)\frac{e^{\alpha |x|}}{|x|^{\frac{d-1}2}}.$$
In dimension $d\geq3$ the right side is always non-negative for $\alpha^2\leq c^2$. We can then simply take $\alpha=c$ and obtain similarly as before that 
$$\beta_\mu-u_\mu\leq \frac{e^{c(|x|-R_\mu+h)}}{|x|^{\frac{d-1}2}}.$$
Note that the bound is blowing up at the origin but it gives us~\eqref{eq:exp_decay} for $r\geq1$, with $C=e^{ch}$. For $r\leq1$ we can simply use that, since $u_\mu$ is decreasing,
$$\beta_\mu-u_\mu\leq \beta_\mu-u_\mu(1)\leq e^{c(1-R_\mu+h)}\leq e^{c(1+h)}e^{c(|x|-R_\mu)}.$$
Finally, upon increasing again the constant $C$ to cover the interval $[R_\mu-h,R_\mu+h]$ where $u_\mu$ is bounded by 1, we obtain~\eqref{eq:exp_decay} in dimensions $d\geq3$.

As usual, the two-dimensional case requires a bit more care. When $d=2$ we introduce $w_\mu:=\sqrt{r}(\beta_\mu-u_\mu)$ which satisfies
$$\left(-w_\mu''+\frac{c^2}4w_\mu\right)\leq\left(\frac{1}{r^2}-3c^2\right) \frac{w_\mu}{4}\leq 0,\qquad\text{on $\left[\frac{1}{\sqrt3 c},R_\mu-h\right]$.}$$
The function 
$$z_\mu:=e^{-\frac{c}2r}\left(w'_\mu+\frac{c}2 w_\mu\right)=e^{-cr}\left(e^{\frac{c}2r}w_\mu\right)'$$
satisfies $z'_\mu\geq0$ on $\big[1/(\sqrt3 c),R_\mu-h\big]$ and therefore we find
$$\left(e^{\frac{c}2r}w_\mu\right)'\leq  z_\mu(R_\mu-h)e^{cr}\qquad \text{for all}\quad \frac{1}{\sqrt3 c}\leq r\leq R_\mu-h.$$
Integrating over $\big[1/(\sqrt3 c),R_\mu-h\big]$ and using that $z_\mu$ and $w_\mu$ are increasing, we obtain
$$\beta_\mu-u_\mu(r)\leq  \left(e^{1/\sqrt{3}}+\frac{1}{c}\right)\frac{z_\mu(R_\mu-h)}{\sqrt{r}}e^{\frac{c}2r}\qquad \text{for all}\quad \frac{1}{\sqrt3 c}\leq r\leq R_\mu-h.$$
Using that $u_\mu'(R_\mu-h)$ and $u_\mu(R_\mu-h)$ are bounded, we have 
$$z_\mu(R_\mu-h)\leq C\sqrt{R_\mu-h}\;e^{-\frac{c}2(R_\mu-h)}$$
for some constant $C$, and hence we have shown the bound
$$\beta_\mu-u_\mu(r)\leq \left(e^{1/\sqrt{3}}+\frac{1}{c}\right)  C e^{\frac{c}2(r-R_\mu+h)}\qquad \text{for all}\quad \frac{1}{\sqrt3 c}\leq r\leq R_\mu-h.$$
Increasing the constant $C$ to get the bound on $[0,(\sqrt3 c)^{-1}]\cup[R_\mu-h,R_\mu+h]$, we obtain~\eqref{eq:exp_decay} in dimension $d=2$ as well.

Next we turn to the derivatives. The equation~\eqref{eq:ODE} can also be rewritten in the form
\begin{equation}
 \left(r^{d-1}u_\mu' \right)'=-r^{d-1}g_\mu(u_\mu).
 \label{eq:ODE2}
\end{equation}
After integrating over $[r,\ii)$ and on $[0,r]$, this gives the estimate on $|v_\mu'(r)|$ in~\eqref{eq:exp_decay_derivative} after using~\eqref{eq:exp_decay} together with
$$|g_\mu(v)|\leq \begin{cases}C|v| & \text{for $0\leq v\leq \gamma$,}\\
C(\beta_\mu-v) &\text{for $\gamma\leq v\leq \beta_\mu$.}
\end{cases}$$
Note that for $r\leq1$ we have the more precise bound
\begin{equation}
 |u_\mu'(r)|=\frac1{r^{d-1}}\int_0^rs^{d-1}|g_\mu(u_\mu(s))|\,ds\leq Cre^{-R_\mu}.
 \label{eq:exp_decay_derivative_origin}
\end{equation}
For the second derivative we can therefore use the equation~\eqref{eq:ODE} to obtain the corresponding bound in~\eqref{eq:exp_decay_derivative}.
\end{proof}

We are now able to prove the convergence~\eqref{eq:limit_mu_mu_*} of the statement, even though we still do not know the behavior of $R_\mu$ in terms of $\mu$.

\begin{lemma}[Global convergence]\label{lem:CV_u_derivative}
We have the uniform convergence
\begin{equation}
 \lim_{\mu\to\mu_*}\norm{u_\mu-U_*(\cdot-R_\mu)}_{L^\ii(\R_+)}=0
 \label{eq:CV_Lii}
\end{equation}
and the convergence of the derivatives
\begin{equation}
 \lim_{\mu\to\mu_*}\norm{u_\mu'-U_*'(\cdot-R_\mu)}_{L^p(\R_+)}=0,
 \label{eq:CV_derivative}
\end{equation}
in $L^p(\R_+)$ for all $1\leq p\leq\ii$.
\end{lemma}

\begin{proof}
The exponential bounds from Lemma~\ref{lem:exp_bounds} give $|v'_\mu(r)-U_*'(r)|\leq C e^{-c|r|}$ for $r\geq-R_\mu$. From the dominated convergence theorem, this shows that 
$$\lim_{\mu\to\mu_*}\int_{-R_\mu}^\ii|v'_\mu(r)-U_*'(r)|^p\,dr=\lim_{\mu\to\mu_*}\int_{0}^\ii|u'_\mu(r)-U_*'(r-R_\mu)|^p\,dr=0.$$
The convergence of the derivatives in $L^1$ together with the fact that $v_\mu(0)=U_*(0)$ imply the uniform convergence~\eqref{eq:CV_Lii}. Finally, the uniform convergence for $v_\mu'-U_*'$ follows from the similar $L^1$ convergence of $v_\mu''-U_*''$ on $[-R_\mu,\ii)$. 
\end{proof}

Next we expand the mass in terms of $R_\mu$. 

\begin{lemma}[Expansion of the mass]\label{lem:CV_L2}
We have 
\begin{equation}
\int_0^\ii r^{\alpha}u_\mu(r)^2\,dr=\frac{\beta_*^2}{\alpha+1}R_\mu^{\alpha+1}+O(R_\mu^{\alpha})+O\big(R_\mu^{\alpha+1}(\mu_*-\mu)\big),
\label{eq:claim_L2c}
\end{equation} 
for all $\alpha>0$. At $\alpha=d-1$ this implies
\begin{equation}
M(\mu)=\norm{u_\mu}^2_{L^2(\R^d)}=\frac{|\bS^{d-1}|\beta_*^2}{d}R_\mu^d+O(R_\mu^{d-1})+O\big(R_\mu^d(\mu_*-\mu)\big).
\label{eq:claim_L2}
\end{equation} 
\end{lemma}

We will prove later that $R_\mu\sim C(\mu_*-\mu)^{-1}$ so that the two errors on the right side are actually of the same order. 

\begin{proof}
Recall that $\beta_\mu$ is the second root of $g_\mu=g_*+(\mu_*-\mu)u$. From the implicit function theorem, we obtain 
\begin{equation}
 \beta_\mu=\beta_*+\frac{\beta_*}{g'(\beta_*)}(\mu-\mu_*)+O\big((\mu-\mu_*)^2\big).
 \label{eq:approx_beta_mu}
\end{equation}
Using the upper bound~\eqref{eq:exp_decay} we then find
\begin{align*}
\int_0^\ii r^\alpha u_\mu(r)^2\,dr&\leq u_\mu(0)^2\int_0^{R_\mu} r^\alpha \,dr+C\int_{0}^\ii (r+R_\mu)^\alpha e^{-cr}\,dr \\
&\leq \frac{u_\mu(0)^2}{\alpha+1}R_\mu^{\alpha+1}+O(R_\mu^\alpha)\\
&=\frac{\beta_*^2}{\alpha+1}R_\mu^{\alpha+1}+O\left(R_\mu^{\alpha+1}(\mu_*-\mu)\right)+O(R_\mu^\alpha).
\end{align*}
The second estimate is because $u_\mu(0)\leq \beta_\mu\leq\beta_*+C(\mu_*-\mu)$.
Using the lower bound in~\eqref{eq:exp_decay}, we also obtain
\begin{align*}
\int_0^\ii r^\alpha u_\mu(r)^2\,dr&\geq \int_{0}^{R_\mu} r^\alpha \left(\beta_\mu-Ce^{-c(R_\mu-r)}\right)^2\,dr \\
&=\frac{\beta_*^2}{\alpha+1}R_\mu^{\alpha+1}+O(R_\mu^{\alpha})
\end{align*}
since $\beta_\mu\geq\beta_*$. This gives the stated expansion~\eqref{eq:claim_L2c}, hence~\eqref{eq:claim_L2} after passing to spherical coordinates.
\end{proof}

Finally, we obtain the exact behavior of $R_\mu$ in terms of $\mu_*-\mu$. 

\begin{lemma}[Behavior of $R_\mu$]\label{lem:R_mu}
We have
\begin{equation}
R_\mu\underset{\mu\to\mu_*}\sim \frac{2\sqrt2(d-1)}{\beta_*^2(\mu_*-\mu)}\int_0^{\beta_*} \sqrt{|G_*(s)|}\,ds.
\label{eq:R_mu}
\end{equation}
\end{lemma}

\begin{proof}
We integrate~\eqref{eq:integrated_ODE} and obtain
\begin{equation*}
\frac12\int_0^\ii u_\mu'(r)^2\,dr-(d-1)\int_0^\ii\int_r^\ii\frac{u_\mu'(s)^2}{s}\,ds\,dr+\int_0^\ii G_\mu(u_\mu(r))\,dr=0.
\end{equation*}
After integrating by parts we find
$$-(d-1)\int_0^\ii\int_r^\ii\frac{u_\mu'(s)^2}{s}\,ds\,dr=-(d-1)\int_0^\ii u'_\mu(r)^2\,dr$$
and the Pohozaev-type relation
\begin{equation}
\left(d-\frac32\right)\int_0^\ii u_\mu'(r)^2\,dr=\int_0^\ii G_\mu(u_\mu(r))\,dr.
\label{eq:Pohozaev}
\end{equation}
We split the second integral in the form
\begin{multline*}
\int_0^\ii G_\mu(u_\mu(r))\,dr=\int_{-R_\mu}^\ii G_\mu(v_\mu(r))\,dr\\=\int_{-R_\mu}^0 \big(G_\mu(v_\mu(r))-G_\mu(\beta_\mu)\big)\,dr+\int_{0}^\ii G_\mu(v_\mu(r))\,dr+R_\mu G_\mu(\beta_\mu).
\end{multline*}
When $\mu$ is in a neighborhood of $\mu_*$, we have the bounds
$$|G_\mu(v)|\leq Cv^2\qquad\text{for all $v\in[0,\gamma]$}$$
and
$$|G_\mu(v)-G_\mu(\beta_\mu)|\leq C(v-\beta_\mu)^2\qquad\text{for all $v\in[\gamma,\beta_\mu]$}.$$
With the exponential bounds~\eqref{eq:exp_decay}, this allows us to use Lebesgue's dominated convergence theorem and deduce that 
$$\lim_{\mu\to\mu_*}\int_{0}^\ii G_\mu(v_\mu(r))\,dr=\int_0^{\ii}G_*(U_*(r))\,dr,$$
$$\lim_{\mu\to\mu_*}\int_{-R_\mu}^0 \big(G_\mu(v_\mu(r))-G_\mu(\beta_\mu)\big)\,dr=\int_{-\ii}^0G_*(U_*(r))\,dr.$$
Using the $L^2$ convergence~\eqref{eq:CV_derivative} of $u'_\mu$, we therefore obtain from~\eqref{eq:Pohozaev}
\begin{align}
\lim_{\mu\to\mu_*}R_\mu G_\mu(\beta_\mu)&=\left(d-\frac32\right)\int_{-\ii}^\ii U_*'(r)^2\,dr-\int_{-\ii}^\ii G_*(U_*(r))\,dr\nn\\
&=\sqrt{2}(d-1)\int_{0}^{\beta_*}|G_*(s)|^{\frac12}\,ds.\label{eq:limit_R_G}
\end{align}
Using the expansion of $\beta_\mu$ in~\eqref{eq:approx_beta_mu}, this implies
$$G_\mu(\beta_\mu)=G_*(\beta_\mu)+\frac{\mu_*-\mu}{2}\beta_\mu^2=\frac{\mu_*-\mu}{2}\beta_*^2+O\big((\mu-\mu_*)^2\big)$$
and after inserting in~\eqref{eq:limit_R_G} we obtain~\eqref{eq:R_mu}.
\end{proof}

Inserting~\eqref{eq:R_mu} into~\eqref{eq:claim_L2} we obtain immediately that 
\begin{equation}
M(\mu) =\frac{\Lambda}{(\mu_*-\mu)^d}+O\left((\mu_*-\mu)^{-d+1}\right)_{\mu\to\mu_*}
\label{eq:behavior_M}
\end{equation}
with the constant $\Lambda$ introduced in the statement. This is the first limit in Theorem~\ref{thm:limit_mu_star}. It only remains to prove the convergence for $M'(\mu)$, which requires a detailed analysis of the linearized operator. 

\subsection{The linearized operator}
The difficulty with the derivative $M'(\mu)=-2\langle u_\mu,(\cL_\mu)_{\rm rad}^{-1}u_\mu\rangle$ is that the first eigenvalue of $\cL_\mu$ tends to zero. Indeed, since $u_\mu-U_*(|x|-R_\mu)\to0$ in $L^\ii(\R^d)$, as proved before in Lemma~\ref{lem:CV_u_derivative}, the intuition is that the operator $\cL_\mu$ behaves like $-\Delta-g_*'(0)=-\Delta+\mu_*$ at infinity and like $-\Delta-g_*'(\beta_*)$ close to the origin. These are two positive operators. On the other hand, the restriction to the radial sector behaves like the operator
$$L_*:=-\frac{\rm d^2}{{\rm d}r^2}-g'_*(U_*)\qquad\text{on $L^2(\R)$,}$$
in the neighborhood of $|x|=R_\mu$. This is because in this region the $d$-dependent term $-(d-1)({\rm d}/{\rm d}r)/r$ becomes negligible.
Note that we have
$L_*U_*'=0$
which proves that $L_*\geq0$ with $\ker(L_*)={\rm span}(U_*')$ .
On the other hand, there is a spectral gap above $\lambda_1(L_*)=0$ since the essential spectrum starts at $\min\left(\mu_*,-g'_*(\beta_*)\right)>0$ and the first eigenvalue is always simple, by the Perron-Frobenius theorem. 
From this discussion, we conclude that the first eigenvalue $\lambda_1(\mu)$ of the operator $\cL_\mu$ should tend to 0, that the corresponding eigenfunction $\phi_\mu$ should behave like $U_*'(\cdot -R_\mu)$ and that $\cL_\mu$ should have a uniform spectral gap above its first eigenvalue, when restricted to the radial sector. The following result confirms this intuition. 

\begin{lemma}[The linearized operator in the limit $\mu\to\mu_*$]\label{lem:linearized}
The lowest eigenvalue of $\cL_\mu$ behaves as
\begin{equation}
\lambda_1(\mu)=-\frac{d-1}{(R_\mu)^2}+o\left(\frac1{(R_\mu)^2}\right)_{\mu\to\mu_*}
\label{eq:lambda_1}
\end{equation}
and the corresponding normalized positive eigenfunction $\phi_\mu$ satisfies
\begin{equation}
\lim_{\mu\to\mu_*}\norm{\phi_\mu+\frac{U_*'(\cdot-R_\mu)}{\kappa(R_\mu)^{\frac{d-1}{2}}}}_{L^2(\R^d)}=0
\label{eq:CV_L2_phi_mu}
\end{equation}
with
$$\kappa=2^{\frac14}|\bS^{d-1}|^{\frac12}\left(\int_{0}^{\beta_*}\sqrt{|G_*(v)|}\,dv\right)^{\frac12}.$$
In addition, we have the bound
\begin{equation}
P_\mu^\perp(\cL_\mu)_{\rm rad}P_\mu^\perp\geq cP_\mu^\perp
\label{eq:estim_cL_orthogonal}
\end{equation}
for some $c>0$ where $P_\mu^\perp=1-|\phi_\mu\rangle\langle\phi_\mu|$ is the projection on the orthogonal of $\phi_\mu$, within the sector of radial functions. 
\end{lemma}

\begin{proof} We split the proof into several steps.

\medskip

\noindent\textbf{Step 1. Upper bound on $\lambda_1(\mu)$.}
We recall that $u_\mu'$ satisfies the equation
$$\cL_\mu u_\mu'+\frac{d-1}{|x|^2}u_\mu'=0.$$
By the variational principle, this proves immediately that 
$$\lambda_1(\mu)\leq \norm{u_\mu'}_{L^2(\R^d)}^{-2}\pscal{u_\mu',\cL_\mu u_\mu'}=-(d-1)\frac{\int_0^\ii r^{d-3}u_\mu'(r)^2\,dr}{\int_0^\ii r^{d-1}u_\mu'(r)^2\,dr}.$$
We obtain from Lemma~\ref{lem:CV_u_derivative} and the same analysis as in Lemma~\ref{lem:CV_L2} that 
$$\int_0^\ii r^{\alpha}u_\mu'(r)^2\,dr=(R_\mu)^{\alpha}\int_\R U_*'(r)^2\,dr+o(R_\mu^\alpha),\qquad \forall \alpha\geq0$$
and this gives in dimension $d\geq3$ the upper bound
\begin{equation}
 \lambda_1(\mu)\leq -\frac{d-1}{(R_\mu)^2}+o\left((\mu_*-\mu)^2\right).
 \label{eq:upper_bd_lambda_1}
\end{equation}
Dimension $d=2$ requires a little more attention. In this case we can for instance use~\eqref{eq:limit_d_term} which gives 
\begin{align*}
(d-1)\int_0^\ii \frac{u_\mu'(r)^2}{r}\,dr &=G_\mu(u_\mu(0))\\
&=G_\mu(\beta_\mu)+O(e^{-R_\mu})\\
&=\frac{d-1}{R_\mu}\int_{\R}U_*'(r)^2\,dr+o(R_\mu^{-1}).
\end{align*}
In the second line we have used~\eqref{eq:exp_decay} whereas in the third line we have used~\eqref{eq:limit_R_G}. We therefore obtain the same upper bound~\eqref{eq:upper_bd_lambda_1} in dimension $d=2$. 

\medskip

\noindent\textbf{Step 2. Convergence.}
Let $\psi_\mu$ be the first (radial positive) eigenfunction of $\cL_\mu$, normalized in the manner
$$\psi_\mu(R_\mu)=-U_*'(0)=\sqrt{-2G_*(\gamma)}.$$
The function $\psi_\mu$ solves the linear equation
$$\left(-\Delta-g_\mu'(u_\mu)-\lambda_1(\mu)\right)\psi_\mu=0.$$
Usual elliptic regularity gives that $\psi_\mu$ is bounded in $C^2(B_{R_\mu+h}\setminus B_{R_\mu-h})$ for every fixed $h>0$.
Since $\lambda_1(\mu)<0$ we also obtain 
$$\left(-\Delta-g_\mu'(u_\mu)\right)\psi_\mu\leq 0\qquad \text{on $\R^d$}.$$
By arguing exactly as in the proof of Lemma~\ref{lem:exp_bounds}, we can then obtain a uniform bound in the form
\begin{equation}
0< \psi_\mu(r)\leq Ce^{-c|R_\mu-r|} 
 \label{eq:exp_decay_psi_mu}
\end{equation}
for some constants $C,c>0$. Using the equation 
$$\left(r^{d-1}\psi_\mu'\right)'=-r^{d-1}\left(\lambda_1(\mu)+g_\mu(u_\mu)\right)\psi_\mu$$
and the fact that $|\lambda_1(\mu)|\leq \|g_\mu'\|_{L^\ii(]0,1[)}$ is bounded, we can deduce that 
$$|\psi'_\mu(r)|\leq \begin{cases}
Ce^{-c|R_\mu-r|}& \text{on $[0,\ii)$,}\\
Cre^{-cR_\mu} & \text{on $[0,1]$}.
                     \end{cases}$$
After extracting a subsequence, we can thus assume that $\psi_{\mu_n}(\cdot+R_{\mu_n})\to V$ strongly in $L^1\cap L^\ii(\R)$ and $\lambda_1(\mu_n)\to\lambda$, with $-V''-g'_*(U_*)V=\lambda V$. Since $V\geq 0$ and $V(0)=-U_*'(0)>0$ we must then have $V=-U_*'$ and $\lambda=0$. We have therefore proved that $\lambda_1(\mu)\to0$ and $\psi_\mu(\cdot+R_\mu)\longrightarrow -U_*'$ in $L^1\cap L^\ii(\R)$.

\medskip

\noindent\textbf{Step 3. Lower bound on $\lambda_1(\mu)$.} 
Next we derive the lower bound on $\lambda_1(\mu)$. Since $u_\mu'$ has a constant sign, this shows that it must be the first eigenvector of the operator $\cL_\mu +(d-1)|x|^{-2}$.  In other words, we have 
\begin{equation}
 \cL_\mu +\frac{d-1}{|x|^2}\geq0
 \label{eq:compare_cL_mu}
\end{equation}
in the sense of quadratic form. Hence we can use that 
$$\lambda_1(\mu)=\frac{\pscal{\psi_\mu,\cL_\mu\psi_\mu}}{\norm{\psi_\mu}^2_{L^2(\R^d)}}\geq-(d-1)\frac{\int_0^\ii r^{d-3}\psi_\mu(r)^2\,dr}{\int_0^\ii r^{d-1}\psi_\mu(r)^2\,dr}.$$
The pointwise bounds~\eqref{eq:exp_decay_psi_mu} allow us to pass to the limit exactly as for $u_\mu'$ and conclude that 
$$\int_0^\ii r^{\alpha}\psi_\mu(r)^2\,dr=(R_\mu)^{\alpha}\int_\R U_*'(r)^2\,dr+o(R_\mu^\alpha),\qquad \forall \alpha\geq0.$$
This gives the desired lower bound in dimensions $d\geq3$. 

The proof does not quite work in dimension $d=2$, since in this case $\int_{0}^\ii r^{-1}\psi_\mu(r)^2\,dx=+\ii$. Instead, we choose a radial localization function $\chi_\mu\in C^\ii_c(\R^2)$ so that $\chi_\mu\equiv1$ on the ball of radius $R_\mu-2\sqrt{R_\mu}$, $\chi_\mu\equiv0$ outside of the ball of radius $R_\mu-\sqrt{R_\mu}$ with $|\nabla\chi_\mu|\leq C/\sqrt{R_\mu}$  and we set $\eta_\mu:=\sqrt{1-\chi_\mu^2}$. The IMS localization formula tells us that 
\begin{align*}
\lambda_1(\mu)\norm{\psi_\mu}_{L^2(\R^2)}^2&= \int_{\R^2}|\nabla \psi_\mu|^2-\int_{\R^2}g_\mu'(u_\mu)\psi_\mu^2\\
&= \pscal{\chi_\mu\psi_\mu,\cL_\mu\chi_\mu\psi_\mu}+\pscal{\eta_\mu\psi_\mu,\cL_\mu\eta_\mu\psi_\mu}\\
&\qquad\qquad-\int_{\R^2}(|\nabla \chi_\mu|^2+|\nabla \eta_\mu|^2)\psi_\mu^2\\
&\geq-(d-1)|\bS^1|\int_{0}^\ii\frac{\eta_\mu^2\psi_\mu^2}{r}\,dr-Ce^{-c\sqrt{R_\mu}}\\
&\geq-\frac{(d-1)|\bS^1|}{R_\mu-2\sqrt{R_\mu}}\int_{0}^\ii\psi_\mu^2\,dr-Ce^{-c\sqrt{R_\mu}}.
\end{align*}
The convergence in $L^2$ allows us to conclude. In the first inequality we have used that $g_\mu'(u_\mu(r))<0$ for $r\leq R_\mu-\sqrt{R_\mu}$ since $u_\mu$ is (exponentially) close to $\beta_\mu$ in this range and $g_\mu'(\beta_\mu)<0$. This implies $\pscal{\chi_\mu\psi_\mu,\cL_\mu\chi_\mu\psi_\mu}\geq0$. We have also used~\eqref{eq:compare_cL_mu} for the second term and the exponential bound~\eqref{eq:exp_decay_psi_mu} for the localization error. 

The normalized eigenfunction in the statement is $\phi_\mu=\norm{\psi_\mu}_{L^2(\R^d)}^{-1}\psi_\mu$ with
\begin{align*}
\norm{\psi_\mu}^2_{L^2(\R^d)}&=R_\mu^{d-1}|\bS^{d-1}|\int_{\R}U_*'(r)^2\,dr+o(R_\mu^{d-1})\\
&=R_\mu^{d-1}|\bS^{d-1}|\sqrt2\int_{0}^{\beta_*}\sqrt{|G_*(v)|}\,dv+o(R_\mu^{d-1}) 
\end{align*}
and the result~\eqref{eq:CV_L2_phi_mu} follows.

\medskip

\noindent\textbf{Step 4. Lower bound on the orthogonal to $\phi_\mu$.} 
Let us now prove~\eqref{eq:estim_cL_orthogonal}. We can argue by contradiction and assume that there exists a subsequence $\mu_n\to\mu_*$ such that $\lambda_2(\mu_n)\to0$, where $\lambda_2(\mu_n)$ is then the second eigenvalue of $\cL_\mu$ within the sector of radial functions. But the exact same arguments as before then give that the corresponding eigenfunction $\tilde\psi_\mu$ satisfies $\tilde\psi_\mu(\cdot+R_\mu)\to cU_*'$, which cannot hold because the functions have to be orthogonal to each other. Hence we must have $\liminf_{\mu\to\mu_*}\lambda_2(\mu)>0$. This concludes the proof of Lemma~\ref{lem:linearized}.
\end{proof}

We finally use Lemma~\ref{lem:linearized} to derive the stated limit~\eqref{eq:limit_M_mu_infty} for $M'(\mu)$. We write 
\begin{align*}
M'(\mu)&=-2\pscal{u_\mu,(\cL_\mu)^{-1}u_\mu}\\
&=-2\left(\frac{|\pscal{u_\mu,\phi_\mu}|^2}{\lambda_1(\mu)}+\pscal{P_\mu^\perp u_\mu,(\cL_\mu)^{-1}P_\mu^\perp u_\mu}\right)\\
&=-2\frac{|\pscal{u_\mu,\psi_\mu}|^2}{\lambda_1(\mu)\norm{\psi_\mu}_{L^2(\R^d)}^2}+O(R_\mu^{d}),
\end{align*}
where we have used $P_\mu^{\perp}(\cL_\mu)^{-1}P_\mu^{\perp}\leq P_\mu^\perp /c$ to infer
$$\pscal{P_\mu^\perp u_\mu,(\cL_\mu)^{-1}P_\mu^\perp u_\mu}\leq \frac{2}{c}M(\mu)=O(R_\mu^{d})$$
by Lemma~\ref{lem:CV_L2}. From the previous convergence properties, we have
$$\pscal{u_\mu,\psi_\mu}=-|\bS^{d-1}|R_\mu^{d-1}\int_{\R} U_*U_*'+o(R_\mu^{d-1})=\frac{|\bS^{d-1}|\beta_*^2}{2}R_\mu^{d-1}(1+o(1))$$
and hence we obtain the limit for $M'(\mu)$ in the statement from~\eqref{eq:lambda_1}. This concludes the proof of Theorem~\ref{thm:limit_mu_star}.\qed

\section{Proof of Theorem~\ref{thm:prop_I_lambda} on the variational principle $I(\lambda)$}\label{sec:prop_I_lambda}
It is classical that $I(\lambda)\leq0$ and that $I(\lambda)$ is non-increasing. First we prove that $I$ is concave. Letting $v(x)=u(\lambda^{-1/d}x)$ which satisfies $\int_{\R^d}|v|^2=1$ whenever  $\int_{\R^d}|u|^2=\lambda$ we can rewrite
$$I(\lambda)=\lambda\, J\!\left(\lambda^{-\frac2d}\right)$$
where
\begin{multline}
J(\eps):=\inf_{\substack{v\in H^1(\R^d)\cap L^{p+1}(\R^d)\\ \int_{\R^d}|v|^2=1}}\bigg\{\frac{\eps}{2}\int_{\R^d}|\nabla v(x)|^2\,dx+\frac1{p+1}\int_{\R^d}|v(x)|^{p+1}\,dx\\
-\frac1{q+1}\int_{\R^d}|v(x)|^{q+1}\,dx\bigg\} 
\label{eq:def_J}
\end{multline}
The function $J$ is non-decreasing, concave and non-positive and this implies that $I$ itself is concave. This is because we have
\begin{equation}
I''(\lambda)=-\frac{2(d-2)}{d^2}\lambda^{-\frac2d-1} J'\!\left(\lambda^{-\frac2d}\right)+\frac4{d^2}\lambda^{-\frac4d-1} J''\!\left(\lambda^{-\frac2d}\right)\leq0 
\label{eq:I_concave}
\end{equation}
in the sense of distributions on $(0,\ii)$. From the concavity of $I$ we deduce that there exists a unique $\lambda_c$ such that $I\equiv0$ on $[0,\lambda_c]$ and $I$ is strictly decreasing (and hence negative) on $(\lambda_c,\ii)$. In dimension $d\geq3$ we even see from~\eqref{eq:I_concave} that $I$ is strictly concave on $(\lambda_c,\ii)$. 

The proof that there exists a minimizer for all $\lambda>\lambda_c$ is very classical. By rearrangement inequalities~\cite{LieLos-01}, we can restrict the infimum to radial-decreasing functions. If $\{u_n\}$ is a minimizing sequence consisting of such functions for $I(\lambda)$, then we can assume after passing to a subsequence that $u_n\to u$ weakly in $H^{1}(\R^d)\cap L^{p+1}(\R^d)$ and strongly in $L^r(\R^d)$ for all $2<r<\min(p+1,2^*)$ with $2^*=2d/(d-2)$ when $d\geq3$ and $2^*=\ii$ when $d=2$, by Strauss' compactness lemma for radial functions~\cite{Strauss-77,BerLio-83}. In particular, $u_n\to u$ strongly in $L^{q+1}(\R^d)$. By Fatou's lemma we then have 
$$I(\lambda')\leq \cE(u)\leq \liminf_{n\to\ii}\cE(u_n)=I(\lambda),$$
where $\lambda'=\int_{\R^d}u^2$. Since $I(\lambda)<I(\lambda')$ for $\lambda'<\lambda$ when $\lambda>\lambda_c$, we must then have $\lambda'=\lambda$ and the convergence is strong in $L^2(\R^d)$. In particular, $u$ is a minimizer. 

Any minimizer for $I(\lambda)$, when it exists, can be chosen positive radial-decreasing after rearrangement. It solves~\eqref{powernl} for some $\mu$. From the Pohozaev identity~\eqref{pohozaev} we obtain
$$\frac{d-2}{2d}\int_{\R^d}|\nabla u(x)|^2\,dx=\int_{\R^d}G_\mu(u(x))\,dx=-I(\lambda)-\frac{\mu\lambda}{2}+\frac12\int_{\R^d}|\nabla u(x)|^2\,dx$$
and hence
\begin{equation}
 \mu=\frac2\lambda\left(-I(\lambda)+\frac1d \int_{\R^d}|\nabla u(x)|^2\,dx\right)
 \label{eq:relation_mu_I_kinetic}
\end{equation}
which is strictly positive since $I(\lambda)\leq0$ (except of course in the trivial case $\lambda=0$ where $u\equiv0$ is the only solution). Since $u\geq0$ is radial-decreasing, it must therefore coincide with the unique corresponding $u_\mu$.

Next we look at $\lambda=\lambda_c$. For $q<1+4/d$ a simple scaling argument shows that $I(\lambda)<0$ for all $\lambda>0$, hence $\lambda_c=0$ in this case. The unique minimizer is then $u_0=0$.

For $q\geq 1+4/d$, it is useful to characterize $\lambda_c$ through the inequality~\eqref{eq:Gagliardo-Nirenberg-type}. We have by definition 
\begin{equation}
\frac1{q+1}\int_{\R^d}|u(x)|^{q+1}\,dx\leq \frac12\int_{\R^d}|\nabla u(x)|^2\,dx+\frac1{p+1}\int_{\R^d}|u(x)|^{p+1}\,dx
\end{equation}
for all $u\in H^1(\R^d)\cap L^{p+1}(\R^d)$ such that $\int_{\R^d}|u|^2\leq\lambda_c$. Replacing $u$ by $\ell^{d/2}u(\ell\cdot)$ and optimizing over $\ell$, we obtain
\begin{multline}
\int_{\R^d}|u(x)|^{q+1}\,dx\leq (q+1)\frac{dp-d-4}{2}\times\\
\times\left(\frac{1}{d(p-q)}\int_{\R^d}|\nabla u(x)|^2\,dx\right)^{1-\theta}\left(\frac{2}{(p+1)(dq-d-4)}\int_{\R^d}|u(x)|^{p+1}\,dx\right)^{\theta}
\end{multline}
for all $u\in H^1(\R^d)\cap L^{p+1}(\R^d)$ such that $\int_{\R^d}|u|^2\leq\lambda_c$, with
$$\theta:=\frac{q-1-\frac{4}d}{p-1-\frac{4}d}\in(0,1).$$
When $\theta=0$ the formulas have to be extended by continuity in an obvious manner but the corresponding optimal $\ell$ vanishes. 
This gives
\begin{multline}
\norm{u}_{L^{q+1}(\R^d)}^{q+1}\leq  (q+1)\frac{dp-d-4}{2}\left(\frac{1}{d(p-q)}\right)^{1-\theta}\left(\frac{2}{(p+1)(dq-d-4)}\right)^{\theta}\times\\
\times\left(\lambda_c^{-\frac12}\norm{u}_{L^2(\R^d)}\right)^{q-1-\theta(p-1)}\norm{\nabla u}_{L^2(\R^d)}^{2(1-\theta)}\norm{u}_{L^{p+1}(\R^d)}^{\theta(p+1)}
\end{multline}
for all $u\in H^1(\R^d)\cap L^{p+1}(\R^d)$. In addition, we have equality everywhere for $u$ a minimizer of $I(\lambda_c)$, when it exists, rescaled in the appropriate manner as above. This shows that the best constant in the Gagliardo-Nirenberg-type inequality
\begin{equation}
\norm{u}_{L^{q+1}(\R^d)}^{q+1}\leq C_{p,q,d}\norm{u}_{L^2(\R^d)}^{q-1-\theta(p-1)}\norm{\nabla u}_{L^2(\R^d)}^{2(1-\theta)}\norm{u}_{L^{p+1}(\R^d)}^{\theta(p+1)}
\label{eq:Gagliardo-Nirenberg-type_bis}
\end{equation}
is exactly given by
$$C_{p,q,d}=(q+1)\frac{dp-d-4}{2}\left(\frac{1}{d(p-q)}\right)^{1-\theta}\left(\frac{2}{(p+1)(dq-d-4)}\right)^{\theta}\lambda_c^{\frac{1+\theta(p-1)-q}2}.$$
From the usual Gagliardo-Nirenberg inequality, it is easily seen that $C_{p,q,d}<\ii$ for $q\geq1+4/d$ and therefore $\lambda_c>0$. In the simpler case $q=1+4/d$ we find 
$$C_{p,1+4/d,d}=\frac{d+2}{d}\lambda_c^{-\frac{2}d}$$
and~\eqref{eq:Gagliardo-Nirenberg-type_bis} is indeed the usual Gagliardo-Nirenberg inequality. The corresponding optimizer is the function $Q$ which solves the NLS equation $-\Delta Q-Q^q+Q=0$ and then
\begin{equation}
\lambda_c=\int_{\R^d}Q(x)^2\,dx, \qquad\text{for $q=1+\frac4d$}.
\label{eq:value_lambda_c}
\end{equation}
From Theorem~\ref{thm:limit_mu_0}, this exactly coincides with $M(0)$. The function $Q$ can however not be an optimizer for $I(\lambda_c)$ because the inequality does not involve the $L^{p+1}$ norm, hence $Q$ does not solve the appropriate equation. This is due to the fact that we have to scale it with $\ell\to0$ as above. As a conclusion we have proved that $\lambda_c>0$ for all $q\geq1+4/d$ and that there cannot be an optimizer for $I(\lambda_c)$ at $q=1+4/d$.

It remains to show the existence for $\lambda=\lambda_c$ and $q>1+4/d$. Let us consider a sequence $\lambda_n\searrow \lambda_c$ and call $u_n=u_{\mu_n}$ a sequence of corresponding radial-decreasing minimizers. The sequence of multipliers $\mu_n$ cannot tend to 0 because $\cE(u_{\mu_n})=I(\lambda_n)<0$ and we know that $\cE(u_\mu)$ is always positive for $\mu$ close to the origin by Corollary~\ref{cor:Energy}. On the other hand it can also not converge to $\mu_*$ because there $\cE(u_\mu)$ is unbounded. Hence after extracting a subsequence we have $\mu_n\to\mu\in(0,\mu_*)$ and $u_{\mu_n}\to u_\mu$ strongly in $(H^1\cap L^\ii)(\R^d)$. The function $u_\mu$ is the sought-after minimizer. 

If $0\leq\lambda<\lambda_c$, then there cannot be a minimizer $u$. If there was one $u$ (positive and radial-decreasing without loss of generality), then it would solve~\eqref{powernl} for some $\mu$. Using again~\eqref{eq:relation_mu_I_kinetic} and $I(\lambda)=0$ we find $\mu>0$. Since 
$$\cE\big((1+\eps)u_\mu\big)=-\mu\lambda\eps+o(\eps)$$
we find that $I(\lambda)$ becomes negative on the right of $\lambda$, which can only happen at $\lambda_c$ by definition. This concludes the proof of Theorem~\ref{thm:prop_I_lambda}.\qed

\begin{remark}[Compactness of minimizing sequences]
The strict monotonicity of $J$ in~\eqref{eq:def_J} implies that $I(t\lambda)>tI(\lambda)$ for every $\lambda>\lambda_c$ and every $t\in(0,1)$. This, in turn, implies that $I(\lambda)<I(t\lambda)+I((1-t)\lambda)$. By the concentration-compactness method~\cite{Lions-84,Lions-84b}, these `binding inequalities' imply that all the minimizing sequences for $I(\lambda)$ converge strongly in $H^1(\R^d)$ to a minimizer, up to space translations and a subsequence. 
\end{remark}

\section{Proof of Theorem~\ref{thmuniqnondeg}}\label{sec:proof_thm_uniqueness}

In this section we provide the full proof of Theorem~\ref{thmuniqnondeg}, although we sometimes refer to the literature for some classical parts or to~\cite{LewRot-15} for arguments which coincide with the ones in that paper. Since we are interested in proving the uniqueness and the non-degeneracy of positive radial solution to~\eqref{nleq}, we consider the associated ordinary differential equation
\begin{equation}\label{nlradeq}
\left\{
\begin{aligned}
&u''+\frac{d-1}{r}u'+g(u)=0 \text{ on }\R^*_+\\
&u'(0)=0
\end{aligned}
\right.
\end{equation}
and we focus on showing the uniqueness and non-degeneracy of positive solutions such that $(u(r),u'(r))\to 0$ when $r\to\infty$. This system has a local energy, given by 
\begin{equation}\label{locenergy}
H(r)=\frac{u'(r)^2}{2}+G(u(r)),\qquad  \text{ with } G(\eta)=\int_0^\eta g(s)\,ds,
\end{equation}
which decreases along the trajectories, since 
$$H'(r)=-\frac{d-1}{r}u'(r)^2\leq0.$$ 

We parametrize the solutions $u_y$ to~\eqref{nlradeq} by $u_y(0)=y$. Since we are interested in positive solution with $\|u_y\|_{\infty}<\beta$, then $y<\beta$. Hence, following \cite{McLeod-93,LewRot-15}, we introduce the three sets
\begin{align*}
&S_+=\{y\in (0,\beta): \min_\R u_y >0\},\\
&S_0=\{y\in (0,\beta): u_y>0 \text{ and } \lim_{r\to +\infty}u_y(r)=0\},\\
&S_-=\{y\in (0,\beta): u_y(r_y)=0 \text{ for some (first) } r_y>0\},
\end{align*}
which form a partition of $(0,\beta)$. In case $y\in S_0$, we set $r_y=+\infty$. One should think of plotting the solution in the plane $(u',u)$ as in Figure~\ref{fig:portrait}. Then, as we will show, $S_+$ exactly correspond to all the solutions that cross the vertical axis, while staying above the horizontal axis at all times. On the other hand, $S_-$ consists of those crossing the horizontal axis first (we will show they cannot cross the vertical axis before). We are particularly interested in the set $S_0$ containing the remaining solutions which are converging to the point $(0,0)$ at infinity while staying in the quadrant $(u'<0,u>0)$. Our goal is indeed to show that $S_0$ is reduced to one point. A transition between $S_-$ and $S_+$ is typically a point in $S_0$ and this is actually how one can prove the existence of solutions by the shooting method. Here we \emph{assume} the existence of one such solution, hence we have $S_0\neq\emptyset.$
Points in $S_0$ typically occur as transition points between $S_-$ and $S_+$. The main idea of the proof is to show that for any $y\in S_0$, we must have 
\begin{equation}
(y-\eta,y)\in S_+\qquad\text{and}\qquad (y,y+\eta)\in S_-
\label{eq:transition}
\end{equation}
for some sufficiently small $\eta>0$. In other words, there can only exist transitions from $S_-$ to $S_+$ and never the other way around, when $y$ is increased starting from $y=0$. This will imply uniqueness. The way to show~\eqref{eq:transition} is to prove that the variation with respect to the initial condition $y$,
\begin{equation}
v_y:=\frac{\partial}{\partial y}u_y,
\label{eq:def_v}
\end{equation}
tends to $-\ii$ at infinity, as well as its derivative $v_y'$. This implies that the curves move enough to cross either the horizontal or the vertical axis when $y$ is moved a bit, for a sufficiently large $r$. The function $v_y$ in~\eqref{eq:def_v} turns out to be the zero-energy solution of the linearized operator $\Delta+g'(u)$ with $v(0)=1$. The fact that $v_y$ diverges implies $v_y\notin L^2(\R_+,r^{d-1}\,dr)$, which means that the kernel of $\Delta+g'(u)$ cannot contain any non-trivial radial function. It is then classical~\cite{Weinstein-85} that this implies the non-degeneracy~\eqref{eq:non_degenerate}. 

\begin{figure}[t]
\centering
\includegraphics[width=7cm]{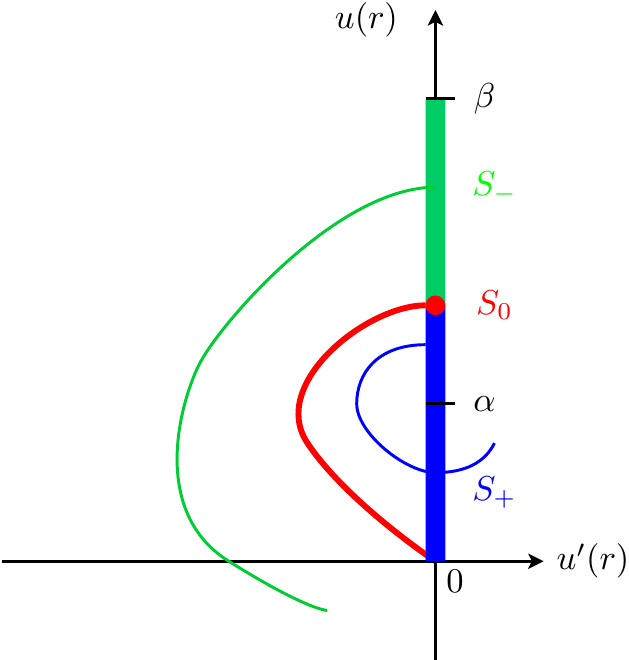}
\caption{Phase portrait of the solutions in the plane $(u',u)$. The interval $(0,\beta)$ is partitioned into the sets $S_+$, $S_-$ and $S_0$.  Solutions with an initial datum in $S_+$ cross first the vertical axis and stay positive for all times, whereas solutions in $S_-$ cross first the horizontal axis. The set $S_0$ contains the solutions that stay in the quadrant for all times and converge to the origin. The goal is to prove that $S_0$ is reduced to one point as in the picture.
\label{fig:portrait}}
\end{figure}

We turn to the proof  of the theorem, which we split into four steps corresponding to Lemmas~\ref{proppartition}--\ref{lem:nondegenerate}, respectively. In the first we show some rather classical facts about the sets $S_-$, $S_+$ and $S_0$.

\begin{lemma}[Properties of the sets $S_+,S_-,S_0$]\label{proppartition}\ 
\begin{enumerate}
\item We have $(0,\alpha]\subset S_+$. 
\item The set $S_-$ is open.
\item\label{proppartition2} If $y\in S_0\cup S_-$, then $u'_y<0$ on $(0,r_y)$. In particular, $u_y$ is strictly decreasing on $(0,r_y)$.
\item\label{proppartition3} If $y\in S_0$, then $\lim_{r\to+\infty} u'_y(r)=0$ and $\int_0^y g(s)\,ds>0$. Moreover, $u_y$ has the following behavior at infinity 
$$
u_y(r)\underset{r\to+\infty}{\sim} C\,\frac{e^{-\sqrt{-g'(0)}r}}{r^{\frac{d-1}{2}}}\quad u'_y(r)\underset{r\to+\infty}{\sim} -\sqrt{-g'(0)}\,C\,\frac{e^{-\sqrt{-g'(0)}r}}{r^{\frac{d-1}{2}}}
$$
for some $C>0$.
\item  If $y\in S_+$, then $u'_y$ vanishes at least once and, for the first positive root $r'_y$ of $u'_y$, we have $H(r'_y)<0$. 
\item The set $S_+$ is open.
\end{enumerate}
\end{lemma}

\begin{remark}
From (\ref{proppartition3}), we see that if $G(\eta)\le 0$ for all $\eta\in (0,\beta)$, then $S_0=\emptyset$. Hence, if $S_0\neq\emptyset$, as we assume in our case, we can define 
$$\gamma:=\inf\{\eta\in(0,\beta), G(\eta)>0\}.$$ 
\end{remark}

\begin{proof}
\textit{(1)} If $y\in (0,\alpha]$, then we have 
$$H(r)\leq H(0)=G(y)=\int_0^yg(t)\,dt<0$$ 
for all $r$ since the local energy $H$ is decreasing along a solution and $g<0$ on $(0,\alpha)$. Therefore $u(r)$ cannot vanish (at a zero of $u$ we would have $H(r)\geq0$). This proves that $(0,\alpha]\subset S_+$.

\smallskip

\noindent \textit{(2)} At any point $r>0$ so that $u_y(r)=0$ we must have $u_y'(r)\neq0$, otherwise $u_y\equiv0$. The implicit function theorem then proves that $r_y$ depends smoothly on the initial condition $y$ and hence $S_-$ is open.

\smallskip

\noindent \textit{(3)} This is~\cite[Lem.~4]{LewRot-15} and the argument goes as follows. For $y\in S_0\cup S_-\subset (\alpha,\beta)$ we have $u''_y(0)=-g(y)/d<0$ since $g$ is positive on $(\alpha,\beta)$. This shows that $u_y'(r)<0$ for small $r>0$. On the other hand, for $y\in S_-$ we have $u_y'(r_y)<0$ since $r_y$ is by definition the first zero of $u_y$. If $u'$ changes sign before $r_y$ then $u_y$ must have a local strict minimum at some point $0<r'<r_y$ and then there must be another point $r''$ at which $u_y(r'')=u_y(r')$. But then 
$$\frac{u_y'(r'')^2}{2}=H(r'')-H(r')=-(d-1)\int_{r'}^{r''}\frac{u_y'(s)^2}{s}\,ds<0,$$
a contradiction. Therefore $u_y$ must vanish before $u'_y$ and the solution crosses first the horizontal axis in the phase portrait. The argument is similar when $r_y=+\ii$. 

\smallskip

\noindent \textit{(4)} If $y\in S_0$, then $u'_y<0$ and for $r$ large enough $u''_y(r)=-\frac{d-1}{r}u'_y(r)-g(u_y(r))>0$. Hence, $u'_y$ has a limit at infinity, which can only be zero since $u$ tends to zero. Next, because of the monotonicity of the energy $H$, 
\begin{equation}
 \int_0^y g(s)\,ds=H(0)>\lim_{r\to+\infty}H(r)=0.
 \label{eq:estim_G}
\end{equation}
Finally, the explicit decay rate of $u_y$ and $u'_y$ is a classical fact whose proof can for instance be found in~\cite{BerLioPel-81}. 

\smallskip

\noindent \textit{(5)} This is \cite[Lem.~5]{LewRot-15}, which follows the presentation in \cite{Frank-13}. If $y=\alpha$, then $u_y\equiv \alpha$ and $H(r)<0$ for all $r\in \R^+$. Hence, let $y\neq \alpha$. First of all, we prove that $u'_y$ must vanish. Otherwise, for all $y\in S_+\cap (0,\alpha)$, $u_y$ is increasing and,  for all $y\in S_+\cap (\alpha,\beta)$, $u_y$ is decreasing. This implies that $\lim_{r\to+\infty }u_y(r)=\alpha$. Next, let $U(r)=r^{\frac{d-1}{2}}(u_y(r)-\alpha)$. The function $U$ solves the equation
$$
U''=\left(\frac{(d-1)(d-3)}{4r^2}- \frac{g(u_y)}{u_y-\alpha}\right)U
$$
and, in the limit $r\to+\infty$, $U''\sim -g'(\alpha) U$ with $g'(\alpha)>0$. This leads to a contradiction. Hence $u'_y$ vanishes and we call $r'_y$ its first root.

To prove $H(r'_y)<0$, we consider two cases. First, if $y\le \alpha$, $H(r'_y)<H(0)<0$. If $y>\alpha$, then $u_y$ has a local minimum at $r'_y$. Hence, from equation~\eqref{nlradeq}, we obtain $g(u_y(r'_y))<0$ which implies $0<u_y(r'_y)<\alpha$. Finally, $H(r'_y)<0$.

\smallskip

\noindent \textit{(6)} To prove that $S_+$ is open, we proceed again as in \cite[Lem.~5]{LewRot-15}. We know that $(0,\gamma)\subset S_+$ ( where $\gamma=\inf\{\eta\in(0,\beta), G(\eta)>0\}$). So let $y\in S_+\cap [\gamma,\beta)$ and $z$ in a small neighborhood of $y$. As a consequence, $u_z$ has a local minimum at $r'_z$ with $H(r'_z)<0$ which implies $G(u_z(r))<0$ for all $r>r'_z$. Hence, there exists $\varepsilon>0$ such that $\varepsilon\le u_z(r)\le \gamma-\varepsilon$. Therefore, $z\in S_+$.
\end{proof}

Let now $v_y$ be the unique solution to the linear problem 
\begin{equation}\label{nlradeqlin}
\left\{
\begin{aligned}
&L(v):=v''+\frac{d-1}{r}v'+g'(u_y)v=0 \text{ on }(0,\ii),\\
&v(0)=1,\\
&v'(0)=0.
\end{aligned}
\right.
\end{equation}
The function $v_y$ is the variation of $u_y$ with respect to the initial condition $u_y(0)=y$, that is, $v_y=\partial_yu_y$. We have the following proposition on the solution to~\eqref{nlradeqlin}, which is the core of the proof of the theorem.

\begin{lemma}[Solution of the linearized problem]\label{proplinear}Let $y\in S_0$. Then 
\begin{enumerate}
\item $v_y$ vanishes exactly once.
\item $v_y(r)$ and $v_y'(r)$ diverge exponentially fast to $-\infty$ as $r\to+\infty$.
\end{enumerate}
\end{lemma}

\begin{proof}
This is exactly \cite[Lem. 7, Lem. 8]{LewRot-15} and we will not reproduce all the details here. The argument is based on the Wronskian identity 
\begin{equation}\label{eqwronskian}
(r^{d-1}(v_y f' - fv'_y))'=r^{d-1}v_y L(f)
\end{equation}
for different choices of the test function $f$, where $L(f)$ is defined in~\eqref{nlradeqlin}. In particular, a simple calculation shows that 
\begin{align*}
L(u_y)&=u_yg'(u_y)-g(u_y)\\
L(u_y')&=\frac{d-1}{r^2}u'_y
\end{align*}
and
\begin{equation*}
L(ru_y')=-2g(u_y).
\end{equation*}

\smallskip

\noindent \textit{(1)} We argue by contradiction and assume that $v_y>0$. Using $f=u_y'<0$ provides 
$$(r^{d-1}(v_y u_y'' - u_y'v'_y))'=(d-1)r^{d-3}u'_yv_y<0$$ 
since $v_y>0$ and $u'_y<0$. This shows that $r^{d-1}(v_y u_y'' - u_y'v'_y)=r^{d-1}v_y^2(u_y'/v_y)'$ is decreasing and since it vanishes at the origin, $u'_y/v_y$ is decreasing. Since this function is negative close to the origin, we have $0<v_y\leq -c u'_y$ for some $c>0$ and all $r\geq 1$.  In addition, $r^{d-1}(v_y u_y'' - u_y'v'_y)\leq -c$ for (another) $c>0$ and all $r\geq 1$. But $|v_y u_y''|\leq c|u_y''|\,|u_y'|$ decreases exponentially at infinity and hence $r^{d-1}u_y'v'_y\geq c/2$ for $r$ large enough. From the behavior at infinity of $u'_y$ this proves that $-v'_y\geq c/2r^{(1-d)/2}e^{\sqrt{g'(0)}r}$ for large $r$, which shows that $v'_y\to-\ii$, a contradiction. 

The proof that it vanishes only once is the same as in~\cite[p~p. 357--358]{Tao-06}, deforming solutions starting from the constant solution $u\equiv \alpha$ at $y=\alpha$ and using that there are no double zeroes.

\smallskip

\noindent \textit{(2)} For the proof of the central fact that $v_y,v'_y\to-\ii$, we call $r_*$ the unique zero of $v_y$ at which we must have $v_y'(r_*)<0$. We then take in the Wronskian identity $f=u_y+cru'_y$ where $c=-u_y(r_*)/(r_*u'_y(r_*))>0$ is chosen so that $f(r_*)=0$. Then we obtain
$$(r^{d-1}(v_y f' - fv'_y))'=r^{d-1}v_y\Big(u_yg'(u_y)-(1+2c)g(u_y)\Big)=r^{d-1}v_y I_\lambda(u_y)$$
with $\lambda=1+2c$. The function $r^{d-1}(v_y f' - fv'_y)$ vanishes both at $0$ and at $r_*$, hence its derivative must vanish at least once in $(0,r_*)$. At this point $\tilde r$  we have $I_\lambda(u_y(\tilde r))=0$ and since $I_\lambda$ is assumed to have only one root over $(0,\beta)$, this shows that $I_\lambda(u_y)>0$ hence $(r^{d-1}(v_y f' - fv'_y))'<0$ on $(r_*,\ii)$. The argument is then similar as above to show that $v_y,v'_y\to-\ii$, see~\cite[Lem.~8]{LewRot-15}.
\end{proof}

Note that Lemma~\ref{proplinear} shows that, if $u_y$ is a positive radial solutions to~\eqref{nleq} which vanishes at $+\infty$ then it is \emph{radial} non-degenerate. Indeed, since the unique solution $v_y$ to~\eqref{nlradeqlin} diverges exponentially fast when $r\to\ii$, the kernel of $L$ as an operator on $L^2(\R_+,r^{d-1}dr)$ is trivial. That the kernel in the full space $L^2(\R^d)$ is spanned by the partial derivatives of $u$ will be proved later in Lemma~\ref{lem:nondegenerate}. 

At this point we have all the tools for proving the uniqueness of positive solutions to~\eqref{nlradeq}. This is done by using the following third proposition.

\begin{lemma}\label{propisolated} Let $y\in S_0$. Then there exists $\varepsilon>0$ such that $(y-\varepsilon,y)\subset S_+$ and $(y,y+\varepsilon)\subset S_-$.
\end{lemma}

As above, the proof goes as in~\cite[Lem. 3(b)]{McLeod-93}.

\begin{proof} Let $y\in S_0$. Choose $a>0$ such that $g'(s)\le g'(0)/2$ for all $s\in[0,a)$ and $\bar R$ such that $u_y(r)\le a$ for all $r\ge \bar R$. Such an $\bar R$ exists since for all $y\in S_0$, $u_y$ is strictly decreasing and goes to $0$ at $+\infty$. Moreover, thanks to Lemma~\ref{proplinear}, we can choose $R\ge \bar R$ such that $v_y(R)<0$ and $v'_y(R)<0$.

As a consequence, there exists $\varepsilon>0$ such that for all $z\in (y,y+\varepsilon)$, we have $0<u_z(R)<u_y(R)$ and $u'_z(R)<u'_y(R)<0$. In particular, the function $w:=u_z-u_y$ is such that $w(R)<0$ and $w'(R)<0$. Suppose, by contradiction, $z\in S_0\cup S_+$. Then $w$ must tend to $0$ or become positive at some point. Therefore, the function $w$ must have a negative minimum at $R'>R$. As a consequence, by using~\eqref{nlradeq}, we obtain
\begin{equation*}
w''(R')=-(g(u_z(R'))-g(u_y(R')))=- g'(\theta)w(R')
\end{equation*}
for some $0<u_z(R')<\theta<u_y(R')\le a$. Hence, $g'(\theta)\le g'(0)/2<0$. This leads to $w''(R')<0$ which is a contradiction. As a conclusion, $z\in S_-$. The argument is the same for $z\in (y-\varepsilon,y)$.
\end{proof}

Lemma~\ref{propisolated} implies that any $y\in S_0$ is an isolated point. Since $S_+$ and $S_-$ are open sets, they can only be separated by points in $S_0$. Now, the lemma says that a point in $S_0$ can only serve as a transition between $S_+$ below and $S_-$ above. Hence, there can be only one transition of this type and $S_0$ contains at most one point. This concludes the proof of uniqueness. 

Our last step is classical and consists in showing the non-degeneracy in the whole space $L^2(\R^d)$.

\begin{lemma}[Non-degeneracy in $L^2(\R^d)$]\label{lem:nondegenerate} Let $u$ a positive radial solution to~\eqref{nleq} with $\|u\|_{\infty}<\beta$ and such that $u(|x|)\to 0$ as $|x|\to +\infty$. Let $\mathcal L_{\rm tot}$ the linearized operator at $u$. Hence, for any $v\in L^2(\R^d,\C)$, 
$\mathcal L_{\rm tot} v:=\cL\re(v)+i\cL'\im(v)$
with 
$$\cL=-\Delta- g'(u),\qquad \cL'=-\Delta -\frac{g(u)}{u}.$$
Then we have
$$\ker(\cL)={\rm span}(\partial_{x_1}u,\ldots,\partial_{x_d}u),\qquad \ker(\cL')={\rm span}(u).$$
\end{lemma}  

\begin{proof}
Since $u$ tends to zero at infinity, the two potentials $-g'(u)$ and $-g(u)/u$ are uniformly bounded. Therefore the operators $\cL$ and $\cL'$ are self-adjoint on $L^2(\R^d)$, with domain $H^2(\R^d)$ and  form domain $H^1(\R^d)$. Moreover they satisfy the Perron-Frobenius property, that their first eigenvalue, when it exists, is necessarily simple with a positive eigenfunction, by the Perron-Frobenius theorem~\cite{ReeSim4}. Finally, any positive eigenfunction is necessarily the first one. 

Since $u$ is a positive solution to~\eqref{nleq}, then $\cL'u=0$. Then $0$ is the first eigenvalue of $\cL'$ and it is non-degenerate, hence $\ker(\cL')={\rm span}(u)$.

The argument for $\cL$ follows~\cite{Weinstein-85}. First of all, we decompose $L^2(\R^d)$ in angular momentum sectors as 
$$L^2(\R^d)=\bigoplus\limits_{\ell\ge 0} L^2(\R_+,r^{d-1}dr)\otimes \mathcal K_\ell$$ with $K_\ell$ the $\ell$th eigenspace of the Laplace-Beltrami operator on the sphere $\mathbb{S}^{d-1}$. Next, since $\cL$ commutes with space rotations, it can be written as $\cL=\bigoplus\limits_{\ell\ge 0} A^{(\ell)}\otimes \mathds{1}$ where
\begin{equation*}
A^{(\ell)}v=-v''-\frac{d-1}{r}v'+\frac{\ell(\ell+d-2)}{r^2}v-g'(u)v
\end{equation*} 
with Neumann boundary condition at $r=0$ for $\ell=0$ and Dirichlet boundary condition for $\ell\ge 1$. By the variational principle, the first eigenvalue of $A^{(\ell)}$ in increasing with $\ell$. Each $A^{(\ell)}$ has the Perron-Frobenius property in $L^2(\R_+,r^{d-1}\,dr)$. Now, the translation-invariance gives $u'\in \ker(A^{(1)})$ and, thanks to Lemma~\ref{proppartition2}, $u'<0$. Therefore $0$ is the first eigenvalue of $A^{(1)}$ and $\ker(A^{(1)})={\rm span}(u')$. Next, for any $\ell \ge 2$, the first eigenvalue of $A^{(\ell)}$ must be positive since $\lambda_1(A^{(\ell)})>\lambda_1(A^{(1)})$, hence  $\ker(A^{(\ell)})=\{0\}$ for $\ell\ge 2$. Finally, it remains to determine $\ker(A^{(0)})$. But $-A^{(0)}$ is simply the operator $L$ defined in~\eqref{nlradeqlin} and we have shown in Lemma~\ref{proplinear} that the kernel of $L$ as an operator on $L^2(\R_+,r^{d-1}dr)$ is trivial. As a conclusion, $\ker(\cL)={\rm span}(\partial_{x_1}u,\ldots,\partial_{x_n}u)$. 
\end{proof}

This concludes the proof of Theorem~\ref{thmuniqnondeg}.\qed

\appendix
\section{Non-degeneracy of $u_0$}\label{app:u_0}

Let $d\geq3$ and $p>q>1$. Let $u_0$ be the unique positive radial solution to the equation
$$-\Delta u_0+u_0^p-u_0^q=0$$
which decays like $u_0(r)\sim C_0r^{2-d}$ at infinity~\cite{BerLio-83,MerPel-90,MerPel-92,KwoMcLPelTro-92}. Define 
$$\cL_0=-\Delta+p(u_0)^{p-1}-q(u_0)^{q-1}$$
the corresponding linearized operator. 

\begin{lemma}[Non-degeneracy of $u_0$]\label{lem:non-degenerate_mu_0}
Let $v$ be the unique solution to
$$\begin{cases}
v''+(d-1)\frac{v'}{r}+qu_0^{q-1}v-pu_0^{p-1}v=0,\\
v(0)=1,\\
v'(0)=0.
  \end{cases}
$$
Then we have $v\notin L^2(\R^d)$, so that 
$$\ker(\cL_0)_{\rm rad}=\{0\},\qquad \ker(\cL_0)=\left\{\partial_{x_1}u_0,...,\partial_{x_d}u_0\right\}.$$
\end{lemma}

\begin{proof}
We assume by contradiction that $v\in L^2(\R^d)$. Then we have the bounds
$$|v(r)|\leq \frac{C}{r^{d-2}},\qquad  |v'(r)|\leq \frac{C}{r^{d-1}},\qquad |v''(r)|\leq \frac{C}{r^{d}}\qquad \forall r\geq1.$$
We have proved in Lemma~\ref{proplinear} that the similar function $v_\mu$ at $\mu>0$ vanishes only once over $(0,\ii)$ and diverges to $-\ii$. From the convergence $g'_\mu(u_\mu)\to g_0'(u_0)$, it follows that $v_\mu\to v$ locally. In particular, $v$ vanishes at most once over $(0,\ii)$. On the other hand, since we have assumed that $v\in L^2(\R^d)$, it has to vanish at least once. This is because we know that $\cL_0$ admits a negative eigenvalue with a positive eigenfunction and that $v$ has to be orthogonal to this eigenfunction. Hence $v$ vanishes exactly once, at some $r_*\in (0,\ii)$. 
Next we follow step by step the argument of Lemma~\ref{proplinear} and use the Wronskian identity with $f=u_0+cru_0'$ with $c=-\frac{u_0(r_*)}{r_*u_0'(r_*)}$. This gives
$$(r^{d-1}(v f' - fv'))'=r^{d-1}v I_\lambda(u_0)$$
with $\lambda=1+2c$. Here
$$I_\lambda(u)=(q-\lambda)u^q-(p-\lambda)u^p$$
vanishes at most once over $(0,\ii)$. Since $r^{d-1}(v f' - fv'_0)$ vanishes both at $0$ and $r_*$, this proves that $I_\lambda$ must vanish on the left of $r_*$, hence has a constant sign on $(r_*,\ii)$. One difference with the case $\mu>0$ is that this sign is unknown, it depends whether $\lambda$ is on the left or right of $q$. In any case, we obtain that $r^{d-1}(v f' - fv')$ is either  increasing or decreasing over $(r_*,\ii)$, and vanishes at $r_*$. 
This cannot happen because $v$ and $f$ decay like $r^{2-d}$ at infinity, whereas their derivatives $v'$ and $f'$ decay like $r^{1-d}$, so that $r^{d-1}(v f' - fv')$ always tends to 0 at infinity.

The rest of the argument for the sectors of positive angular momentum is identical to that of Lemma~\ref{lem:nondegenerate}, of which we use the notation. We know that $\ker(A^{(1)})={\rm span}\{u'_0\}$ and this corresponds to $\partial_{x_1}u_0,...,\partial_{x_d}u_0$ being in the kernel of $\cL_0$. On the other hand, we have $A^{(\ell)}=A^{(1)}+\frac{\ell(\ell +d-2)-d+1}{r^2}> A^{(1)}$ for $\ell\geq2$, which proves that $\ker(A^{(\ell)})=\{0\}$ for $\ell\geq2$. 
\end{proof}


\end{document}